\documentclass[11pt]{amsart}
\usepackage{amsmath,amssymb,latexsym,soul,cite,mathrsfs,breqn}

\usepackage{color,enumitem,graphicx}
\usepackage[colorlinks=true,urlcolor=blue,
citecolor=red,linkcolor=blue,linktocpage,pdfpagelabels,
bookmarksnumbered,bookmarksopen]{hyperref}
\usepackage[english]{babel}

\usepackage[left=2.9cm,right=2.9cm,top=2.8cm,bottom=2.8cm]{geometry}
\usepackage[hyperpageref]{backref}

\usepackage[colorinlistoftodos]{todonotes}
\makeatletter
\providecommand\@dotsep{5}
\def\listtodoname{List of Todos}
\def\listoftodos{\@starttoc{tdo}\listtodoname}
\makeatother

\numberwithin{equation}{section}
\def\cal{\mathcal}
\newtheorem{lemma}{Lemma}
\newtheorem{proposition}{Proposition}
\newtheorem{theorem}{Theorem}
\newtheorem{corollary}{Corollary}
\newtheorem{remark}{Remark}

\title[ %Quasilinear Schr\"odinger equation with almos critical nonlinearity
]
{%Multiplicity of positive solutions for the  almost critical quasilinear Schr\"odinger equation \\
Multiplicity of positive solutions for a quasilinear Schr\"odinger equation
with an almost critical nonlinearity}

\author[G. M. Figueiredo]{Giovany M. Figueiredo}
\author[U. B. Severo]{Uberlandio B. Severo}
\author[G. Siciliano]{Gaetano Siciliano}

\address[G. M. Figueiredo]{\newline\indent Faculdade de Matem\'atica
\newline\indent 
Universidade Federal do Par\'a
\newline\indent
66075-110, Bel\'em, PA, Brazil}
\email{\href{mailto:giovany@ufpa.br}{giovany@ufpa.br}}

\address[U. B. Severo]{
\newline\indent 
Departamento de Matem\'atica
\newline\indent 
Universidade Federal da Para\'\i ba
\newline\indent
58051-900, Jo\~ao Pessoa, PB, Brazil}
\email{\href{mailto: uberlandio@mat.ufpb.br}{uberlandio@mat.ufpb.br}}

\address[G. Siciliano]{\newline\indent Departamento de Matem\'atica
\newline\indent 
Instituto de Matem\'atica e Estat\'istica
\newline\indent 
 Universidade de S\~ao Paulo 
\newline\indent 
Rua do Mat\~ao 1010,  05508-090 S\~ao Paulo, SP, Brazil }
\email{\href{mailto:sicilian@ime.usp.br}{sicilian@ime.usp.br}}

\thanks{G. M. Figueiredo was partially
supported by  CNPq, FAPDF and CAPES, Brazil. U. B. Severo was partially supported by 
CNPq grant 308735/2016-1, Brazil. G. Siciliano  was partially supported by
Fapesp, CNPq and CAPES, Brazil. }

\begin{document}

\maketitle

\begin{abstract}
In this paper we prove an existence result of multiple positive solutions
for the following quasilinear problem
\begin{equation*} 
\left\{
\begin{array}[c]{ll}
-\Delta u - \Delta (u^2)u = |u|^{p-2}u & \mbox{ in } \Omega \\
u= 0 &\mbox{ on } \partial\Omega,
\end{array}
\right.
\end{equation*}
where $\Omega$ is a smooth and bounded domain in $\mathbb R^{N},N\geq3$.
More specifically we prove that, for $p$ near the critical exponent $22^{*}=4N/(N-2)$,
the number of positive solutions is estimated below by  topological invariants
of the domain $\Omega$: the Ljusternick-Schnirelmann category 
and the Poincar\'e polynomial.
\end{abstract}

%********************************************************************************
\section{Introduction}
%********************************************************************************

It is well-known that the 
 general quasilinear Schr\"odinger
 equation
 \begin{equation*}%\label{SCH}
i\partial_t \psi =-\Delta
\psi+V(x)\psi-\widetilde{h}(|\psi|^2)\psi- \kappa\Delta
[\rho(|\psi|^2)]\rho'(|\psi|^2)\psi,
\end{equation*}
where $\psi:\mathbb{R}\times  \mathbb{R}^N \rightarrow\mathbb{C}$ is the unknown,
$\kappa$ is a real constant,
serves as  models for several 
physical phenomena depending on the form of the 
given potential $V=V(x)$ and the given nonlinearities $\rho(s)$ and $\widetilde h$.

In the case $\rho(s) = (1+s)^{1/2}$, the equation %(\ref{SCH}) 
models the self-channeling of a high-power ultra short
laser in matter, see \cite{Borovskii-Galkin,Ritchie}.
It also appears in fluid mechanics
\cite{Kosevich-Ivanov}, in the theory of Heisenberg ferromagnets
and magnons \cite{Takeno-Homma}, in dissipative quantum mechanics
and in condensed matter theory \cite{Makhankov-Fedyanin}. 
%Quasilinear equations of
%the form (\ref{SCH}) have been studied in relation with some
%mathematical models in physics.

When $\rho(s)=s$, 
which is the case we are interested here, the
above equation reduces to
\begin{equation}\label{SCH2}
i\partial_t \psi =-\Delta \psi+V(x)\psi- \kappa\Delta
[|\psi|^2]\psi-\widetilde{h}(|\psi|^2)\psi.
\end{equation}
It was shown that a system describing the self-trapped electron on
a lattice can be reduced in the continuum limit to (\ref{SCH2})
and numerics results on this equation are obtained in
\cite{Brizhik}. In \cite{Hartmann}, motivated by the nanotubes and
fullerene related structures, it was proposed and shown that a
discrete system describing the interaction of a 2-dimensional
hexagonal lattice with an excitation caused by an excess electron
can be reduced to (\ref{SCH2}); moreover numerics results have been done
on domains of disc, cylinder or sphere type. The
superfluid film equation in plasma physics  has also the structure
(\ref{SCH2}), see \cite{Kurihura}.

The search of 
standing wave solutions  $\psi(t,x)=\exp(-iFt)u(x),\; F\in \mathbb{R}$ of \eqref{SCH2}
under a power type nonlinearity $\widetilde h$ reduces
the equation  to
\begin{equation}\label{simp}
-\Delta u -\Delta (u^2)u +W(x)u= h(u)\\
\end{equation}
where $W(x)=V(x)-F$ as $V(x)$,
$h(u)=\widetilde{h}(u^2)u$ and we have assumed, without loss of
generality, that $\kappa =1$.

The quasilinear equation (\ref{simp}) in the whole
$\mathbb{R}^{N}$ has received special attention in the past
several years and various devices have been used:
the method of Lagrange multipliers, which gives a solution 
with an unknown multiplier 
$\lambda$ in front of the nonlinear term (see \cite{Poppenberg-Schmitt-Wang})
and the remarkable change of variable to get a semilinear 
equation in  appropriate Orlicz space framework 
(see \cite{Jeanjean-Colin,do O-Severo,Liu-Wang II}). We refer the reader also to the papers
 \cite{OMS,Liu-Wang I,Liu-wang-wang, RS} and
references therein. 

\medskip

Here we are interested in a special case of \eqref{simp}, that is, we study
the equation in a smooth and bounded domain  $\Omega\subset \mathbb R^{N},N\geq3$,
with constant potential $V(x)=F$ (hence $W(x)=0$)
and with homogeneous Dirichlet boundary conditions; in other words,
 we are interested in the search of positive solutions for the problem
\begin{equation}\label{principal} 
\left\{
\begin{array}[c]{ll}
-\Delta u - \Delta (u^2)u = |u|^{p-2}u & \mbox{ in } \Omega \\
u= 0 &\mbox{ on } \partial\Omega,
\end{array}
\right.
\end{equation}
with $p\in (4,22^{*})$. % and where we consider for simplicity the case of constant  potential $V(x)=F$.
As usual $2^{*}=2N/(N-2)$ is the Sobolev critical exponent of the embedding of $H^{1}_{0}(\Omega)$
into Lebesgue spaces, and  $22^{*}$ turns out to be the critical exponent for the problem,
as it is shown in \cite{Jeanjean-Colin}.
% after using the above mentioned change of variable.

The main goal of this paper is to show that
for $p$ near the critical exponent $22^{*}$, the topology of the domain
influences the number of positive solutions
%In fact we are able to relate the number of
%positive solutions of \eqref{principal} with the ``topology'' of the domain,
in the sense of Theorem \ref{th:main}  and Theorem \ref{th:mainMorse} below.

 Before to state our main results we recall that if $Y$ is a closed
set of a topological space $X$, we denote the
Ljusternik-Schnirelmann category of $Y$ in $X$ by $\mbox{cat}_X (Y)$,
which is the least number of closed and contractible sets in $X$
that cover $Y$. Moreover, $\mbox{cat} \, X$ denotes $\mbox{cat}_X (X)$. 
Then we have the first multiplicity result.
\begin{theorem}\label{th:main} 
There exists $\overline{p} \in (4,22^{*})$ such that for any $p\in [\overline{p},22^{*})$, problem 
\eqref{principal} has at least $cat \,\Omega$ positive weak solutions.
Moreover if $\Omega$ is not contractible in itself then 
\eqref{principal} has at least $cat \,\Omega+1$ positive weak solutions.
\end{theorem}

By implementing the Morse theory we are able to prove also the following
multiplicity result.
Here $\mathcal P_{t}(\Omega)$ is the Poincar\'e polynomial of $\Omega$,
whose definition we recall later.

\begin{theorem}\label{th:mainMorse}
There exists $\overline{p} \in (4,22^{*})$ such that for any $p\in [\overline{p},22^{*})$, problem 
\eqref{principal} has at least $2\mathcal P_{1}(\Omega) -1$ positive solutions, possibly counted with their multiplicity.
\end{theorem}
In whole this paper, a function
$u:\Omega\rightarrow \mathbb{R}$ is called a (weak) solution of
(\ref{principal}) if $u\in H^{1}_{0}(\Omega)\cap
L^{\infty}_{loc}(\Omega)$ and satisfies
\begin{equation*}
\label{solucao fraca}
\int_{\Omega} \Big[(1+2|u|^2)\nabla
u\nabla\varphi +2|\nabla u|^2 u\varphi  \Big]
=\int_{\Omega}|u|^{p-2}u\varphi\,\,\, \mbox{for all} \,\,\, \varphi\in
C_0^{\infty}(\Omega).
\end{equation*}
We point out that, among the solutions we find there is  the ground state, called $\mathfrak g_{p}$, that is the
solution with minimal energy $\mathfrak m_{p}$ in the sense specified in Section \ref{sec:Nehari}.
\medskip

It is worth to mention that  problem  \eqref{principal}
has been studied recently in \cite{Li-Wei}
in a bounded domain and the authors prove,  by using Morse theory,  existence results of multiple solutions.
However the number of solutions found in \cite{Li-Wei} is not in relation with the topology of the domain $\Omega$,
and nothing is said on the sign of the solutions.

So our paper is the first one to relate the number of positive solutions to the
topology of the domain when the exponent is near the critical one $22^{*}$.

\medskip

\subsection{The approach and the main ideas}
Our approach in proving Theorem \ref{th:main} and Theorem \ref{th:mainMorse}
is variational; indeed we first
 use the change of
variable $u=f(v)$
introduced by  \cite{Jeanjean-Colin} to transform problem \eqref{principal}
 into 
\begin{equation*}%\label{cambiointro}
\left\{
\begin{array}[c]{ll}
 -\Delta v =|f(v)|^{p-2}f(v)f'(v) & \text{in } \Omega \medskip \\
 v=0 & \text{on }  \partial{\Omega}.               
\end{array}
\right. 
\end{equation*}
Then we find its solutions as critical points of a $C^{1}$ functional on
the so called Nehari manifold, which is a natural constraint.
In particular we show that the functional on the Nehari manifold is bounded 
below, achieves the {\sl ground state level} $\mathfrak m_{p}$, for $p\in (4,22^{*})$,
and by means of the Ljusternik-Schnirelmann and Morse theories  we prove 
the multiplicity results.

We say that most of all the results we prove in the next two sections are fundamental in order
to achieve the Proposition \ref{baricentri} in Section \ref{sec:Nehari}, which is a key step in order to employ
the Ljusternick-Schnirelmann theory.

\medskip

As it is usual by using  a variational approach,
at some point  it will be important to have a compactness condition, that we recall
here once for all for the reader's convenience.
If $H$ is an Hilbert space, $\mathcal N\subset H$ a submanifold and $I:H\to \mathbb R$ a $C^{1}$ functional,
 we say that  %a $C^{1}$ functional $I$ defined on a Hilbert space $H$, 
$I$ satisfies the Palais-Smale condition on $\mathcal N$ at level $a\in \mathbb R$, $(PS)_{a}$ condition for short,
 if  any sequence $\{u_{n}\}\subset \mathcal N$ such that
\begin{equation}\label{eq:psseq}
I(u_{n}) \to a\quad \text{and} \quad (I_{|\mathcal N})' (u_{n})\to 0
\end{equation}
possesses a  subsequence converging to $u\in \mathcal N$.
We will also say that $I_{|\mathcal N}$ satisfies the $(PS)$ condition.

In general a sequence satisfying the  conditions in \eqref{eq:psseq} is named Palais-Smale
sequence at level $a$, or $(PS)_{a}$ sequence for short. If the value $a$ is not really important,
we will simply speak of  $(PS)$ condition and/or $(PS)$ sequences.

 Let us say that, as it will be evident by the method
we use, we will need a representation of the $(PS)$
sequences for the functional related to the critical problem, that is
\begin{equation}\label{eq:introcritical}
\left\{
\begin{array}[c]{ll}
-\Delta v =|f(v)|^{22^{*}-2}f(v)f'(v) & \quad\text{ in } \Omega\\
v=0 & \quad\text{ on } \partial\Omega.
\end{array}
\right.
\end{equation}
%associated to \eqref{cambiointro}.

 This representation for the quasilinear problem
has never appeared in the literature, to the best of our knowledge.
Hence as  a byproduct of our proofs we obtain
  the  representation of the $(PS)$ sequences (known also as {\sl Splitting Lemma},
  see Lemma \ref{lem:splitting})
for the critical problem \eqref{eq:introcritical},
 which may be useful also in other different context.
 
 Furthermore,
concerning the critical case,
we show in Lemma \ref{lem:nonex} a nonexistence result for problem
\eqref{eq:introcritical}  in a star-shaped domain when $p=22^{*}$; this means that
 the exponent $22^{*}$ is critical also with respect
to the  existence of solutions and implies that
 the ground state level $m_{*}$ is not achieved in this case.
 Nevertheless we show that for every domain $\lim_{p\to22^{*}} \mathfrak m_{p} = m_{*}$, see 
 Theorem \ref{prop:limitem*}. We think this last result is interesting of its right.

\medskip

The ideas we use to prove Theorems \ref{th:main} and \ref{th:mainMorse}
 are mainly inspired from that of \cite{BC1,BC3,BCP} where the authors 
consider the model  problem
$$
\left\{
\begin{array}[c]{ll}
-\Delta u + \lambda u = |u|^{p-2}u &\quad \text{in }\Omega, \\
u=0 &\quad \text{on }\partial\Omega
\end{array}
\right.
$$
and ask how the topology of the domain $\Omega$ affects the number of positive solutions
depending on suitable ``limit'' values of the parameters $\lambda,p$.
They introduced new techniques in order to have a ``picture'' of $\Omega$ in a suitable sublevel of the 
energy functional associated to the problem, and then they use the Ljusternick-Schnirelmann and Morse
 theory in order to deduce a multiplicity result.
Actually the authors treat two cases:
\begin{itemize}
\item[(i)] when $p$ is  fixed and  the parameter $\lambda$ 
is made sufficiently large,
\item[(ii)] when $\lambda$ is fixed and the parameter $p$
in the nonlinearity tends to the critical value $2^{*}$
\end{itemize}
and find solutions for $\lambda$ large in the first case, and for $p$ near $2^{*}$
in the second case.

After the  mentioned  papers \cite{BC1,BC3,BCP}, these 
techniques  have been successfully used to prove multiplicity of positive solutions
for equations involving also different   operators then the Laplacian,
and even in presence of a potential.
%which in particular lead to  look for semiclassical states.
However the existing literature mainly concerns with case (i):
many papers appeared
where the parameter $\lambda$ can be moved, after a rescaling, into the potential
or even as a factor which expands the domain $\Omega$.
For more details and results in this direction, we refer the reader to the papers
\cite{Alves, CV3} for the $p-$Laplacian, \cite{AFF} for the magnetic Laplacian in expanding domains,
  \cite{GEdw} for a system  of fractional Schr\"odinger-Poisson type,
   \cite{GGM} for the fractional Laplacian in expanding domains,
\cite{AUG, CGU,CV2} for  quasilinear  operators:
in all these papers multiplicity result, depending on the topology 
of the domain, have been proved for $\lambda$ large.

However  case (ii) in which the role of the parameter is taken by the exponent of the nonlinearity,
that we believe to be very interesting too,
has been much less explored. Indeed this has motivated the present paper.
To the best of our knowledge there are just two other papers (besides \cite{BC1})  which
consider the case when the parameter $p$ approaches the critical exponent
obtaining multiplicity of solutions depending on the topology of the domain:
they are \cite{Sicilia} where the Schr\"odinger-Poisson system is studied and \cite{SF}
where the fractional Laplacian is considered.
%We cite also  \cite{PPS} where where the De Giorgi Gamma-convergence
%is used to obtain concentration and ap

We point out that the ideas of Benci, Cerami and Passaseo in 
\cite{BC1,BC3,BCP} are not immediately applicable
to our problem
due to the fact that there is the change of variable $f$ which has to be treated very
carefully. In fact we need some new properties of the change of variable, which never appeared before,
see %Corollary \ref{monotonicidades} and 
Lemma \ref{lem:outras}.
Moreover, in contrast to the paper \cite{BC1,BC3,BCP} we can not work
on the $L^{p}-$constraint due to the lack of homogeneity in the equation
which does not permit to eliminate the Lagrange multiplier once it appears.

\subsection{Structure of the paper}
The paper is organized as follows.

In Section \ref{sec:variational} we give the variational setting of the problem.
In particular the change of variable given in  \cite{Jeanjean-Colin} is introduced in order
to have a well defined and $C^{1}$ functional whose critical points are exactly
the solutions we are looking for.
%We show also that the operator representing the second derivative of the 
%functional is of type identity minus a compact operator.

In Section \ref{sec:Nehari} we introduce the Nehari manifold associated to the problems
settled in the domain $\Omega$,
in both cases of $p$ subcritical and critical.
This section is quite technical since we need to perform 
projections of nontrivial functions on different Nehari manifolds,
and compare in some sense the Nehari manifolds of the subcritical problem with the Nehari manifold of the critical one.
A ``local'' $(PS)$ condition is proved for the critical case.
Finally, we give also a Splitting Lemma involving the critical problem 
on the whole $\mathbb R^{N}$.

In Section \ref{bary}, the barycenter map {\sl \`a la Benci-Cerami} is introduced and some properties are proved.

In Section \ref{sec:finale} the proof of Theorem \ref{th:main}  is given by using the Ljusternick-Schnirelmann theory.
 
 In Section \ref{sec:Morse}, after recalling some basic notions in Morse theory
 and show that the second derivative of the functional is ``of type''
isomorphism minus a compact operator, we prove
 Theorem \ref{th:mainMorse}.

\subsection{Final comments}
 As a matter of notations, we will use the letters $C, C',\ldots, C_{1}, C_{2}, \ldots $ 
 to denote suitable positive constants
which do not depend on the functions neither on $p$. Moreover their values,
irrelevant for our purpose,  are allowed to change on every estimate.

The letter  $S$ will be deserved for the embedding constant of $H^{1}_{0}(\Omega)$ in $L^{2^{*}}(\Omega)$.

The symbol $o_{n}(1)$ stands for a vanishing sequence.

We will use sometimes the notation $|u|_{p}$ for the usual $L^{p}-$norm
of the function $u$: no confusion should arise for what concerns the underlying domain.

Other notations will be introduced whenever we need.

Finally, without no loss of generality, we assume throughout the paper that $0\in \Omega$.
%------------------------------------------------------------------------------

\section{Variational framework }\label{sec:variational}

As observed in \cite{severo,severo2}, there are some technical
difficulties to apply directly variational methods to the formal
functional associated to \eqref{principal}, which formally should be given by
\[
J_{p}(u) = \frac{1}{2}\int_{\Omega}
(1+2|u|^2)|\nabla u |^2 -\frac{1}{p}
 \int_{\Omega} |u|^{p} .
\]
%\textcolor{red}{colocar noutro lugar: We will use also the notation $|u|_{p}$ for the usual $L^{p}-$norm
%of the function $u$: no confusion should arise for what concerns the domain.}
The main difficulty  related to 
${J}_{p}$ is that it is not  well defined in  the whole
$H_{0}^{1}(\Omega)$. For example, if
$u$ diverges near $0$ as $|x|^{(2-N)/4}$ and  then is glued to a smooth,
radial, and vanishing function,
%$u\in C_0^1(\Omega\backslash \{0\})$ is defined by
%\[
%u(x)=|x|^{(2-N)/4}\ \ \mbox{for}\ \  x\in \Omega\backslash \{0\},
%\]
we have $u\in H_{0}^{1}(B)$, while the function
$|u|^2|\nabla u |^2$ does not belong to $L^1(B)$. Here $B\subset\Omega$ is a ball containing the origin in $\mathbb R^{N}$.

To overcome this difficulty, we use the arguments developed in
%\cite{severo,severo2} which generalize some arguments
%found in Liu, Wang and Wang \cite{Liu-Wang II} and
 Colin-Jeanjean \cite{Jeanjean-Colin}. More precisely, we make
the change of variables $v= f^{-1}(u)$, where $f$ is defined by
\begin{equation}\label{mudanda de variavel}
\begin{array}{cllc}
f'(t)&=&\dfrac{1}{(1+ 2f(t)^{2})^{1/2}} & \;\;\mbox{ on }\;\;  [0,+\infty),\\
 f(t)&=&-f(-t) & \;\;\mbox{ on }\;\; (-\infty,0].
\end{array}
\end{equation}
Therefore, after the change of variables, the functional
$J_{p}$ can be rewritten in the following way
\begin{equation}\label{funcional}
I_{p}(v):= J_{p}(f(v))=
\frac{1}{2}\int_{\Omega} |\nabla v |^2  
%\frac{1}{2}\int_{\Omega} |f(v)|^2  
-\frac{1}{p}
 \int_{\Omega} |f(v)|^{p}
\end{equation}
which is well defined on the  space $H_{0}^{1}(\Omega)$ endowed with the usual norm
$$\|v\|^{2} = \int_{\Omega} |\nabla v|^{2}. $$
%$$
%\|v\| = |\nabla v|_{2} + \inf_{\xi>0}\biggl[\frac{1}{\xi}+\displaystyle\int_{\Omega}|f(\xi v)|^{2}dx\biggl].
%$$
%The reader can find in \cite{severo} and \cite{severo2} more
%details about the function $f$ and the proof that $\|\cdot\|$ is a
%norm in $H_{0}^{1}(\Omega)$. A direct computation implies that $\|\cdot\|$ is an equivalent norm
%to the usual norm of $H^{1}_{0}(\Omega)$.
A straightforward computation
shows that the functional \eqref{funcional}
%$I_{p}:H_{0}^{1}(\Omega) \rightarrow \mathbb{R}$
is of class $C^1$ with
\[
I_{p}'(v)[w]=\int_{\Omega}\nabla v\nabla w -\int_{\Omega}|f(v)|^{p-2}f(v)f'(v)w 
\] for $v,w\in H_{0}^{1}(\Omega)$. Thus, the critical points of $I_{p}$
correspond exactly to the weak solutions of the semilinear problem
\begin{equation}\label{eq:equivalent}%\tag{$D$}
\left\{
\begin{array}[c]{ll}
 -\Delta v =|f(v)|^{p-2}f(v)f'(v) & \text{in } \Omega \medskip \\
 v=0 & \text{on }  \partial{\Omega}.               
\end{array}
\right. 
\end{equation}
This problem has a close relation with problem
(\ref{principal}). In fact, if $v \in
H_{0}^{1}(\Omega)\cap
L^{\infty}_{loc}(\Omega)$ is a critical point of the
functional $I_{p}$, hence  a weak solution of  \eqref{eq:equivalent}, then $u=f(v)$ is a weak solution of
\eqref{principal}. By the same arguments used to prove Proposition 3.6 of \cite{AUG}, we have that each critical point $v$ of $I_{p}$ belongs to $H_{0}^{1}(\Omega)\cap
L^{\infty}(\Omega)$. 

Summing up we are reduced  to find
nontrivial critical points of $I_{p}$.
Actually, as we will see in Section \ref{sec:Morse} where the Morse theory is used,
 the functional is even $C^{2}$.

Now we show some results 
about the change of variable $f:\mathbb R\to \mathbb R$
that are essential in the next sections.

\begin{lemma}[see \cite{severo,severo2}]
\label{Lema f} The function $f$ and its derivative enjoy the
following properties:
\begin{enumerate}[label=(\roman*),ref=\roman*]
\item \label{Lemafi}  $f$ is uniquely defined, $C^2$ and invertible; \smallskip
  \item\label{Lemafii}  $|f'(t)|\leq 1$ for all $t\in \mathbb{R}$; \smallskip
  \item\label{Lemafiii}  $|f(t)|\leq |t|$ for all $t\in \mathbb{R}$;\smallskip
  \item\label{Lemafiv}  $f(t)/t\rightarrow 1$ as $t\rightarrow 0$; \smallskip
  \item\label{Lemafv}  $|f(t)|\leq 2^{1/4}|t|^{1/2}$ for all $t\in   \mathbb{R}$; \smallskip
  \item\label{Lemafvi}  $f(t)/2 < tf'(t) < f(t)$ for all $t > 0$, and the reverse inequalities hold for $t<0$;\smallskip
  \item\label{Lemafvii}  $f(t)/{\sqrt t}\rightarrow a>0$ as $t\rightarrow +\infty$; \smallskip
  \item\label{Lemafviii}  there exists a positive constant $C$ such that
\[
|f(t)| \geq
\begin{cases}
C|t|,\quad & |t| \leq 1 \\
C|t|^{1/2},\quad & |t|  \geq 1;
\end{cases}
\]
\item\label{Lemafix}  $|f(t)f'(t)|\leq 1/2^{1/2}$ for all $t\in \mathbb{R}$.
\end{enumerate}
\end{lemma}

Particularly useful will be the  inequalities
\begin{equation}\label{eq:util}
f(t)^{2}/2\leq f'(t)f(t)t\leq f(t)^{2} \text{ for all }\ t\in \mathbb R.
\end{equation}
simply obtained by \eqref{Lemafvi} of  Lemma \ref{Lema f}.

We  deduce the following:

\begin{corollary}\label{monotonicidades}The following properties involving the function $f$  hold:
\begin{enumerate}[label=(\roman*),ref=\roman*]
    \item \label{monotonicidadei} The function ${f(t)f'(t)}{t^{-1}}$ is decreasing for
    $t>0$;\smallskip
    \item \label{monotonicidadeii} The function ${f(t)^{3}f'(t)}{t^{-1}}$ is increasing for
    $t>0$;\smallskip
    \item \label{monotonicidadeiii} For any $p>4$, the function $|f(t)|^{p-2}f(t)f'(t)t^{-1}$ is increasing for $t>0$.
   \end{enumerate}
\end{corollary}
\begin{proof} By using \eqref{Lemafvi} of Lemma \ref{Lema f}, it is easy to see that $f(t)/t$ is nonincreasing for $t>0$. Thus,
\[
\begin{aligned}
\frac{d}{dt}\left(\frac{f(t)f'(t)}{t}\right)=&\;
\frac{d}{dt}\left(\frac{f(t)}{t}\right)f'(t) +\frac{f(t)}{t}f''(t)<0
\end{aligned}
\]
for all $t>0$, which shows \eqref{monotonicidadei}.

\medskip

To prove \eqref{monotonicidadeii}, we observe that since 
$$f'(t)=\frac{1}{(1+2f^{2}(t))^{1/2}} \quad  \text{ and } \quad
f''(t)=\frac{-2f(t)f'(t)}{(1+2f(t)^{2})^{3/2}}=-2f(t)(f'(t))^{4},$$ we have
\begin{eqnarray*}
\frac{d}{dt}\left(\frac{f(t)^{3}f'(t)}{t}\right)& =&
\frac{3f(t)^{2}(f'(t))^2t
-2f(t)^{4}(f'(t))^{4}t-f(t)^{3}f'(t)}{t^{2}}\\
&\geq&f'(t)f(t)^{2}\frac{3f'(t)t-f'(t)t-f(t)}{t^{2}}\\
\end{eqnarray*}
and therefore, by \eqref{Lemafvi} and \eqref{Lemafix} of Lemma
\ref{Lema f}, we have for all $t>0$,
\[
\frac{d}{dt}\left(\frac{f(t)^{3}f'(t)}{t}\right) \geq f'(t)f(t)^{2}\frac{2f'(t)t-f(t)}{t^{2}}> 0,
\]
which proves \eqref{monotonicidadeii}.

\medskip

Finally  by using  \eqref{monotonicidadeii} and the equality
$$
\frac{|f(t)|^{p-2}f(t)f'(t)}{t}=|f(t)|^{p-4}\frac{f(t)^{3}f'(t)}{t}\,\,\, \mbox{for} \,\,\, t>0
$$
we obtain \eqref{monotonicidadeiii}.
\end{proof}

The next  properties  will be fundamental in the proof of the key Proposition
\ref{baricentri}.

\begin{lemma}\label{lem:outras}
The following hold true:
\begin{enumerate}[label=(\roman*),ref=\roman*]
%\item \label{outrasi} $f^{2}(t)/2\leq f'(t)f(t)t\leq f^{2}(t)$ for all $t\in \mathbb R$; TIRAR
\item\label{outrasiii} $f(t)f'(t)$ is increasing. In particular,
\begin{eqnarray*}
f(\lambda t)f'(\lambda t) \lambda t &\leq& \lambda f(t) f'(t) t, \quad \forall \lambda\in[0,1], \forall t\geq0 ; \smallskip \\
f(\lambda t)f'(\lambda t) \lambda t &\geq& \lambda f(t) f'(t) t, \quad \forall \lambda\geq1, \forall t\geq0. \smallskip
\end{eqnarray*}
\noindent Moreover for $\alpha\geq0$, we have \smallskip
\item\label{outrasiv} $f(\lambda t)^{\alpha}\geq \lambda^{\alpha} f(t)^{\alpha}$, for all $\lambda\in[0,1]$ and $t\geq0$;
\smallskip
\item\label{outrasv} $f(\lambda t)^{\alpha}\leq \lambda^{\alpha/2} f(t)^{\alpha}$, for all $\lambda\in [0,1]$ and $t\geq0$;
\smallskip
\item\label{outrasvi} $f(\lambda t)^{\alpha} \leq \lambda^{\alpha} f(t)^{\alpha}$, for all $\lambda\geq1$ and  $t\geq0$;
\smallskip
\item\label{outrasvii} $f(\lambda t)^{\alpha} \geq \lambda^{\alpha/2} f(t)^{\alpha}$, for all $\lambda\geq1$ and  $t\geq0$.
\end{enumerate}
\end{lemma}
\begin{proof}
%TIRAR :By \eqref{Lemafvi} of  Lemma \ref{Lema f} we easily get \eqref{outrasi}. \medskip
By  using that
$f''(t) = -2f(t) (f'(t))^{4}$ 
(see the proof of \eqref{monotonicidadeii} in Corollary \ref{monotonicidades})
and \eqref{Lemafix} of Lemma \ref{Lema f}, we find that
$$\frac{d}{dt}\Big(f(t)f'(t) \Big)= (f'(t))^{2} -2f(t)^{2} (f'(t))^{4} =(f'(t))^{2} \Big( 1-2f(t)^{2} (f'(t))^{2}\Big)\geq0 $$
proving \eqref{outrasiii}.

\medskip

Of course if  $\lambda=0$ or $t=0$, \eqref{outrasiv} and \eqref{outrasvi} 
are  satisfied.
Now for $t>0$ fixed we have,
in virtue of  \eqref{Lemafvi} of Lemma \ref{Lema f},
$$
\frac{d}{d\lambda} \Big(\frac{f(\lambda t)^{\alpha}} {\lambda^{\alpha}}\Big)=
\frac{\alpha f(\lambda t)^{\alpha-1} \lambda^{\alpha-1} \Big(f'(\lambda t) \lambda t -f(\lambda t)\Big) }{\lambda^{2\alpha}} \leq0
$$
 and hence
$f(\lambda t)^{\alpha} /\lambda^{\alpha}$
is a non-increasing function. This gives   \eqref{outrasiv} and \eqref{outrasvi}.

\medskip

The proof of
\eqref{outrasv} and \eqref{outrasvii} follow in a similar way, by computing the
derivative with respect to $\lambda$ of $f(\lambda t)^{\alpha}/ \lambda^{\alpha/2}$.
\end{proof}

\section{The Nehari manifolds and compactness results}
\label{sec:Nehari}

In this section we study the Nehari manifolds which appear in relation to our problem.
In particular we need to consider, beside problem \eqref{principal}
also some limit cases with the associated  Nehari manifolds.
%  These facts will be used throughout the paper. 

Associated to the functional \eqref{funcional}, that is,
$$I_{p}(v)=\frac12\int_{\Omega}|\nabla v|^{2}
%+\frac{1}{2}\int_{\Omega} f(u)^{2}
-\frac{1}{p}\int_{\Omega} |f(v)|^{p},
$$
we have the set, usually called
the {\sl  Nehari manifold} associated to \eqref{principal},
\begin{equation*}
\label{Nehari}{\mathcal{N}}_{p}=\left\{ v\in H^{1}_{0}(\Omega)\setminus\{0\}:
G_{p}(v)=0\right\}
\end{equation*}
where
$$
G_{p}(v):=I_{p}^{\prime}(v)[v]=\|v\|^{2} -\int_{\Omega}|f(v)|^{p-2}f(v)f'(v)v.
$$
In particular all the critical points of $I_{p}$  lie in $\mathcal N_{p}$.
%Sometimes we will refer to \eqref{Ivinc} as the constraint functional, also
%denoted with $I_{p}|_{\mathcal{N}_{p}}$. 
In the next Lemma we show the basic properties of $\mathcal N_{p}$. 
We present also the proof of some of its properties 
since we were not able to find them in the literature.
\begin{lemma}\label{lemmanehari}
For all $p\in(4,22^{*}]$, we have:
\begin{enumerate}[label=(\roman*),ref=\roman*]
\item\label{lemmaneharii} $\mathcal{N}_{p}$ is a $C^{1}$ manifold; \smallskip
\item\label{lemmanehariii} there exists $c_{p}>0$ such that $\|v\|\geq c_p$ for every $v\in \mathcal{N}_{p}$; \smallskip
\item\label{lemmanehariiii} it holds $\inf_{u\in\mathcal{N}_{p}}I_{p}(u)>0$;  \smallskip

\item\label{lemmanehariiv} for every $v\neq0$ there exists a unique $t_{p}=t_{p}(v)>0$ such that
$t_{p}v\in\mathcal{N}_{p}$; \smallskip

\item\label{lemmanehariv} $\mathcal N_{p}$ is homeomorphic to the unit sphere 
$\mathbb S = \{v\in H^{1}_{0}(\Omega): \|v\|=1\}$;\smallskip

\item\label{lemmaneharivi} the following equalities are true
\[
\inf_{v\in\mathcal{N}_{p}}I_{p}(v) =\inf_{v\neq0}\max_{t>0}I_{p}(tv)=\inf_{g\in\Gamma_{p}} \, \max
_{t\in[0,1]} I_{p}(g(t)),
\]
where
$$
\Gamma_{p}=\left\{g\in C([0,1];H^{1}_{0}(\Omega)) : g(0)=0, I_{p}(g(1))\le0,
g(1)\neq0\right\}.
$$
\end{enumerate}
\end{lemma}
%\begin{remark}\label{rem:2*}
%Before giving the proof, we observe explicitly that we can even allow $p=22^{*}$
%in Lemma \ref{lemmanehari}
%since no compactness condition is used in its proof.
%\end{remark}
\begin{proof} 
Since
\begin{multline*}
G'_{p}(v)[v]= 2 \displaystyle\int_{\Omega}|\nabla v|^{2}
%+ \displaystyle\int_{\Omega}[f'(v)]^{2}[v]^{2}+\displaystyle\int_{\Omega}|f(v)|f''(v)[v]^{2}
%+\displaystyle\int_{\Omega}|f(v)|f'(v)v\\ 
- \displaystyle\int_{\Omega}(p-1)|f(v)|^{p-2}(f'(v))^{2}v^{2}- \\
\displaystyle\int_{\Omega}|f(v)|^{p-2}f(v)f''(v)v^{2}- \displaystyle\int_{\Omega}|f(v)|^{p-2}f(v)f'(v)v
\end{multline*}
%\begin{eqnarray*}
%G'_{p}(v)[v]&=& 2 \displaystyle\int_{\Omega}|\nabla v|^{2}\\
%%+ \displaystyle\int_{\Omega}[f'(v)]^{2}[v]^{2}+\displaystyle\int_{\Omega}|f(v)|f''(v)[v]^{2}
%%+\displaystyle\int_{\Omega}|f(v)|f'(v)v\\ 
%& -& \displaystyle\int_{\Omega}(p-1)|f(v)|^{p-2}(f'(v))^{2}v^{2}- 
%\displaystyle\int_{\Omega}|f(v)|^{p-2}f(v)f''(v)v^{2}- \displaystyle\int_{\Omega}|f(v)|^{p-2}f(v)f'(v)v
%\end{eqnarray*}
and  $G_{p}(v)=0$ if $v\in \mathcal N_{p}$, we obtain
$$
	G'_{p}(v)[v]= -\displaystyle\int_{\Omega}(p-1)|f(v)|^{p-2}(f'(v))^{2}v^{2}- 
	\displaystyle\int_{\Omega}|f(v)|^{p-2}f(v)f''(v)v^{2}+ \displaystyle\int_{\Omega}|f(v)|^{p-2}f(v)f'(v)v.
$$
Using that $f''(t)=-2f(t)(f'(t))^{4}$  and inequalities in 
 \eqref{Lemafix} and \eqref{Lemafvi} of Lemma \ref{Lema f}, we get
\begin{eqnarray*}\label{decidir}
G'_{p}(v)[v]&=&  \displaystyle\int_{\Omega}|f(v)|^{p-2}f(v)f'(v)v
-(p-1)\int_{\Omega}|f(v)|^{p-2}(f'(v))^{2} v^{2}
 +2\displaystyle\int_{\Omega}|f(v)|^{p}(f'(v))^{4}v^{2}\\
 &=& - \int_{\Omega}|f(v)|^{p-2}f'(v)v \Big[ (p-1)f'(v)v -2f(v)^2(f'(v))^{3}v -f(v) \Big]  \nonumber\\
 &\leq &- \int_{\Omega}|f(v)|^{p-2}f'(v)v \Big[ (p-1)f'(v)v -f'(v)v -f(v) \Big]  \nonumber\\
 &=& - \int_{\Omega}|f(v)|^{p-2}f'(v)v \Big[ (p-2)f'(v)v  -f(v) \Big]  \nonumber\\
 &\leq& - \int_{\Omega}|f(v)|^{p-2}f'(v)v \Big[\frac{p-2}{2}f(v)-f(v)\Big]  \nonumber \\
 &=&- \frac{p-4}{2}\int_{\Omega}|f(v)|^{p-2}f(v)f'(v)v.  \nonumber
\end{eqnarray*}
Finally by using \eqref{eq:util}
we  arrive at
$$
G_{p}'(v)[v] \leq - \frac{p-4}{2}\int_{\Omega} |f(v)|^{p-2}f(v)f'(v)v \leq - \frac{p-4}{4}\int_{\Omega}|f(v)|^{p}<0
$$
which proves \eqref{lemmaneharii}.
\medskip

Let $v\in \mathcal N_{p}$. Then, by using successively \eqref{eq:util} and \eqref{Lemafv} of Lemma
\ref{Lema f}, we get
\begin{eqnarray*}
\| v\|^{2} &= &\int_{\Omega} |f(v)|^{p-2} f(v)f'(v)v \\
&\leq& 2^{p/4}\int_{\Omega}|v|^{p/2} \\
%Se si fa cosí, 
%\textcolor{red}{\leq 2^{2^{*}/2} C \| v\|^{p/2}} \\
% la costante dipende da p. Per questo %si é passato per il 2^{*}!!
& \leq& 2^{p/4}  |\Omega|^{\frac{2 2^{*} - p}{22^{*}}}|v|_{2^{*}}^{p/2} \\
&\leq & 2^{p/4} S^{p/2} |\Omega|^{\frac{2 2^{*} - p}{22^{*}}}\|v\|^{p/2}
\end{eqnarray*}
%where $S$ is the embedding constant of $H^{1}_{0}(\Omega)$ in $L^{2^{*}}(\Omega)$.
and hence we infer 
$$\|v\|\geq \Big( \frac{1}{2^{p/4} S^{p/2} |\Omega|^{\frac{2 2^{*} - p}{22^{*}}}}\Big)^{2/(p-4)}=:c_{p}>0$$
which shows \eqref{lemmanehariii}.
%hence
%$$\liminf_{p\to 22^{*}} \|v\| \geq 
%\Big( \frac{1}{2^{2^{*}/2} S_{N}^{2^{*}} }\Big)^{1/(2^{*}-2)}.$$
%\begin{eqnarray}\label{dis}
%\|u_{ p}\|^{2} &\le &\|u_{p}\| + \int_{\Omega} f(u_{p}) f'(u_{p})u_{p} \\
%&=&\int_{\Omega}|f(u_{ p})|^{p-2}f(u_{p}) f'(u_{p})u_{p}\\
%&\leq&\int_{\Omega}|f(u_{ p})|^{p} 
%\leq  \int_{\Omega}|u_{p}|^{p}  \leq C \|u_{p}\|^{p}
%\end{eqnarray}
%where $C$ is a positive constant which can be made independent of $p$.
%(see e.g. \cite{LL})
%Hence
%\begin{equation*}
%\exists\,c>0 \ \ \mbox{ s.t. } \forall\, p\in(4,22^{*})\, :\ \ 0< c\le\|v\|.
%\end{equation*}

\medskip

 On $\mathcal{N}_{p}$  the functional is positive since, by using \eqref{eq:util}
 we have
\begin{eqnarray*}%\label{eq:pos}
I_{p}(v) &=& \frac{1}{2}\int_{\Omega}|f(v)|^{p-2}f(v)f'(v)v - \frac{1}{p}\int_{\Omega}|f(v)|^{p} \nonumber\\
&\geq &\frac{p-4}{4p} \int_{\Omega}|f(v)|^{p} \nonumber\\
&\geq&0. 
\end{eqnarray*}
%where item (6)  of Lemma \ref{Lema f} has been used. 
% we get $I_{p}(u)\geq (\frac{1}{4} - \frac{1}{p}) \int_{\Omega}|f(u)|^{p} \geq 0$.
%and, being $p>4$, it holds
%\begin{equation*}
%I_{p}(u)=\Big(\frac{1}{2} -\frac1p\Big)\int_{\Omega}|\nabla u|^{2} %+\frac{1}{2}\int_{\Omega}f(u)^{2} \\-\frac1p \int_{\Omega}f(u) f'(u)u 
%+\frac{1}{p}\int_{\Omega} |f(u)|^{p-2}f(u)f'(u)u-\frac{1}{p}\int_{\Omega}|f(u)|^{p}.
%\end{equation*}
%by computing $I_{p}(u)-p^{-1}I_{p}'(u)[u]$, and using  Lemma 2.1 PEMS Alves, Figueiredo, Severo,
% we get
%\begin{eqnarray*}
%I_{p}(u)\geq\Big(\frac{1}{2} - \frac1p\Big )\int_{\Omega}|\nabla u|^{2}%+\Big(\frac{1}{2}-\frac1p\Big)\int_{\Omega}f(u)^{2}
%+\Big(\frac{1}{4} -\frac{1}{p}\Big) \int_{\Omega}|f(u)|^{p} \geq0
%\end{eqnarray*}
%hence $I_{p}$ is bounded below and positive on $\mathcal N_{p}$.
%Then recalling that $p>4$, we have 
%$$m_{p}:=\inf_{u\in\mathcal{N}_{p}}I_{p}(u)>0\,.$$
Moreover, for every $v\in {\mathcal N_{p}}$ 
by \eqref{eq:util} %of Lemma \ref{lem:outras} 
and \eqref{Lemafv}
of Lemma \ref{Lema f},
we have
$$\int_{\Omega}|\nabla v|^{2} = \int_{\Omega} |f(v)|^{p -2} f(v) f'(v)v
 \leq \int_{\Omega}|f(v)|^{p} \leq 2^{p/4}\int_{\Omega}|v|^{p/2} \leq C \Big( \int_{\Omega} |\nabla v|^{2}\Big)^{p/2}$$
and then 
\begin{equation}\label{eq:contra}
\int_{\Omega}|\nabla v|^{2} \geq C'>0.
\end{equation}
Then, if  it were $\inf_{v\in\mathcal{N}_{p}}I_{p}(v)=0$ there would exist  $\{v_{n}\} \subset {\mathcal N_{p}}$ 
such that, by using again \eqref{eq:util},
\begin{eqnarray}\label{eq:nonvanishing}
o_{n}(1) &=& {I_{p}}(v_{n})  \\
& = & {I_{p}}(v_{n}) -\frac{2}{p} {I_{p}'}(v_{n})[v_{n}] \nonumber\\
& = &  \frac{p-2}{2p} \int_{ \Omega}|\nabla v_{n}|^{2} 
+ \frac{2}{p}\int_{\Omega} |f(v_{n})|^{p - 2}f(v_{n})f'(v_{n})v_{n}
 -\frac{1}{p}\int_{\Omega}|f(v_{n})|^{p} \nonumber\\
&\geq&   \frac{p-4}{2p} \int_{\Omega}|\nabla v_{n}|^{2}  \nonumber
\end{eqnarray}
which contradicts \eqref{eq:contra} and concludes the proof of \eqref{lemmanehariiii}.

\medskip

Let $v\neq 0$ and, for $t\geq0$ define the map
$$g(t): = I_{p}(tv)= \frac{t^{2}}{2}\int_{\Omega} |\nabla v|^{2} -\frac{1}{p}\int_{\Omega} |f(tv)|^{p}.$$ 
It is easy to see that $g(0) = 0$ and $g(t) = I_{p}(tv)<0$ for suitably large $t$, by 
\eqref{Lemafviii} of Lemma \ref{Lema f}.

Clearly $g'(t)=I'_{p}(tv)[v]=0$ if and only if $tv\in \mathcal N_{p}$.
Moreover, $g'(t)=0$ means
$$\int_{\Omega}|\nabla v|^{2} = \frac{1}{t} \int_{\Omega}| f(tv)|^{p-2}f(tv) f'(tv)v = \int_{\{x\in \Omega : v(x)\neq0\}}
\frac{f(t|v|) ^{p-1} f'(t|v|) |v|^{2}}{t|v|}
$$
and the right hand side is an increasing function in $t$.
Since by \eqref{Lemafv} of Lemma \ref{Lema f}, 
$$\lim_{t\to 0^{+}} \int_{\Omega} t^{-2}|f(tv)|^{p}\leq 2^{p/4}\lim_{t\to0^{+}}  \int_{\Omega} t^{(p-4)/2}|v|^{p/2}=0,$$
we easily see that  $g(t)>0$ for suitably small $t>0.$
Then there is a unique $t_{p}=t_{p}(v)>0$
such that $g'(t_{p}) = 0$ and $g(t_{p}) = \max_{t>0} g(t)$, i.e. $t_{p} v\in \mathcal N_{p}$, proving \eqref{lemmanehariiv}.

\medskip

The proof of \eqref{lemmanehariv} and \eqref{lemmaneharivi} follows by standard arguments.
\end{proof}

\begin{remark}\label{rem:limbaixo}
Actually in \eqref{lemmanehariii} of Lemma \ref{lemmanehari} the constant $c_{p}$ can be made independent on $p$
far away from $4$.
%$\in (4,22^{*}]$: just take
%$$\xi:=\lim_{p\to 22^{*}} c_{p} =\Big( \frac{1}{2^{2^{*}/2} S^{2^{*}} }\Big)^{1/(2^{*}-2)}>0. $$
Indeed it is easily seen that it is possible to take a small $\eta>0$ such that
$$\xi:=\min_{p\in[4+\eta,22^{*}]}c_{p}>0.$$ 
%ja que nao sei se a fun\c c\~ao $p\mapsto c_{p}$
%\'e decrescente. De toda forma posso pegar o minimo no intervalo $[4+\varepsilon,22^{*}]$ : s\'o nao sei se \'e atingido em $p=22^{*}$
%(por isso nao quero pegar o limite.)
%Nesse caso precisa arrumar muitos dos resultados a seguir que vao valer por $p\geq 4+\varepsilon$,
%mas isso nao \'e um problema para gente...pois queremos provar o teorema principal por $p$ perto do $22^{*}$.
%Eu arrumo tudo isso ...
%}
In other words, all the Nehari manifolds $\mathcal N_{p}$ are bounded away from zero,
independently on $p\in[4+\eta,22^{*}]$, i.e.  there exists $\xi>0$ such that
$$\forall p\in [4+\eta,22^{*}], v\in \mathcal N_{p} \Longrightarrow \|v\|\geq \xi.$$

In the remaining part of the paper, the symbol $\eta$ will be deserved for the small positive
constant given above.
\end{remark}

\medskip

The Nehari manifold well-behaves with respect to the $(PS)$ sequences.
Again, since at this stage no compactness condition is involved, we can even state the 
result for $p\in (4,22^{*}]$.
\begin{lemma}\label{lemmaPS}
Let $p\in (4,22^{*}]$ be fixed and $\{v_{n}\}\subset\mathcal{N}_{p}$ be a $(PS)$ sequence
for $I_{p}|_{\mathcal{N}_{p}}$. Then $\{v_{n}\}$ is a $(PS)$
sequence for the free functional $I_{p}$ on $H^{1}_{0}(\Omega)$.
\end{lemma}
%\begin{proof}
%By definition, $\{u_{n}\}\subset\mathcal{N}_{p}$, $I_{p}|_{\mathcal{N}_{p}}(u_{n})$ 
%is bounded and there exist Lagrange multipliers $\{\mu_{n}%
%\}\subset\mathbf{R}$ such that $(I_{p}|_{\mathcal{N}_{ p}})^{\prime
%}(u_{n})=I^{\prime}_{ p}(u_{n})-\mu_{n} G^{\prime}_{ p}(u_{n})\rightarrow0$ in
%$H^{-1}(\Omega)$. Then recalling the definition of $\mathcal{N}_{ p}$ we have
%$$(I_{p}|_{\mathcal{N}_{p}})^{\prime}(u_{n})[u_n]=\mu_{n} G^{\prime}_{ p}(u_{n})[u_{n}]\rightarrow0.$$
%Since $G^{\prime}_{ p}(u_{n})[u_{n}]\neq0$ it follows that the sequence of
%multipliers vanishes and
%\begin{equation*}
%I^{\prime}_{ p}(u_{n})=(I_{ p}|_{\mathcal{N}_{ p}})^{\prime}(u_{n})+\mu_{n}
%G^{\prime}_{ p}(u_{n})\rightarrow0.
%\end{equation*}
%\end{proof}

Now for $p\in(4,22^{*})$ it is known that the free functional $I_{ p}$ 
satisfies the $(PS)$ condition on $H^{1}_{0}(\Omega)$ and also when restricted to
$\mathcal N_{p}$, see e.g.
\cite[Lemma 3.2 and Proposition 3.3]{CGU}.
In addition to the properties listed in Corollary \ref{lemmanehari},
the manifold ${\mathcal{N}}_{p}$ is a natural constraint for $I_{ p}$ in the
sense that any $u\in\mathcal{N}_{p}$ critical point of $I_{p}|_{\mathcal{N}%
_{p}}$ is also a critical point for the free functional $I_{p}$ (see, for instance,  \cite[Corollary 3.4]{CGU}). Hence the
(constraint) critical points we find are solutions of our problem since no
Lagrange multipliers appear.

In particular,
  as a consequence of the $(PS)$ condition we have
\begin{equation*}\label{mp}
\forall\,p\in(4,22^{*})\,:\ \ \mathfrak m_{ p}:=\min_{v\in \mathcal{N}_{ p}} I_{p}(v)=I_{ p}(\mathfrak g_{ p})>0\,,
\end{equation*}
i.e. $\mathfrak m_{ p}$ is achieved on a function, hereafter denoted with $\mathfrak g_{ p}$. 
Since $\mathfrak g_{ p}$ minimizes the energy
$I_p$, it will be called a \emph{ground state}.
Observe that $\mathfrak g_{p}\geq0$ and are indeed positive by the maximum principle.

\begin{remark}\label{rem:nuovo}
We note that if $\{w_{p}\}_{p\in[4+\eta,22^{*}]}\subset H^{1}_{0}(\Omega)$ is such that for all 
$p\in[4+\eta,22^{*}], w_{p}\in \mathcal N_{p}$,
%\item $\{w_{p}\}_{p\in(4,22^{*})}$ is bounded in $H^{1}_{0}(\Omega)$,
% $\{w_{p}\}_{p\in (4,22^{*})}$ is bounded away from zero in $H^{1}_{0}(\Omega)$,
then 
\begin{eqnarray*}
0<\xi \leq\|w_{p}\|^{2} &=& \int_{\Omega}|f(w_{p})|^{p-2}f(w_{p}) f'(w_{p}) w_{p} \\
&\leq & \int_{\Omega} |f(w_{p})|^{p} \\ 
&\leq&  2^{1/4} \int_{\Omega}|w_{p}|^{p/2}\\
%&\leq& C |w_{p}|_{2^{*}}^{p/2}\\
%&\leq& 2^{1/4}|\Omega|^{1/2} |w_{p}|_{p}^{p/2} \textcolor{red}{ERRADA}\\ 
&\leq& C|w_{p}|_{2^{*}}^{p/2}
%&\leq&  C'\|w_{p}\|^{p/2}, 
\end{eqnarray*}
where $C$ can be choosen independent on $p$.  We deduce that 
the sequences
 $\{|w_{p}|_{p}\}_{p\in[4+\eta,22^{*}]} , \{|f(w_{p})|_{p}\}_{p\in [4+\eta,22^{*}]}$ and 
 $\{|w_{p}|_{2^{*}}\}_{p\in [4+\eta,22^{*}]}$
are bounded away from zero.
%$\{|w_{p}|_{2^{*}}\}_{p\in [4+\eta,22^{*}]}$ is bounded away from zero, 
%due to the fact that
%$$0<C'\leq (\xi 2^{-1/4})^{2/p}\leq|w_{p}|_{p/2} \leq |\Omega| ^{\frac{2^{*} - p/2}{2^{*}p/2}}|w_{p}|_{2^{*}}.$$
%with $C''$ independent on $p\in[4+\eta,22^{*}]$.

In particular,
%since $0<\xi \leq\|\mathfrak g_{p}\|$
%\begin{equation*}
%\exists\,c>0 \ \ \mbox{ s.t. } \forall\, p\in(4,22^{*})\, :\ \ 0< c\le\|\mathfrak g_{p}\|,
%\end{equation*}
% that is the family of minimizers $\{\mathfrak g_{p}\}_{p\in (2,22^{*})}$ is  bounded away from zero in $H^{1}_{0}(\Omega)$. 
this  is true
 %in   1. and 2. above are available 
 for the family of ground states $\{\mathfrak g_{p}\}_{p\in[4+\eta,22^{*})}$. This last fact will be used in 
 the next sections and in particular in Proposition \ref{prop:importante}.
\end{remark}

%With analogous computation as in point 2. of Lemma \ref{lemmanehari}
%we can show that the sequence of minimizers $\{v_{ p}\}_{p\in(4,22^{*})}$ is bounded away from
%zero; indeed, since $v_{ p}\in\mathcal{N}_{ p}\,$, 
%\begin{eqnarray*}
%\int_{\Omega}|\nabla v_{p}|^{2} &= &\int_{\Omega} |f(v_{p})|^{p-2} f(v_{p})f'(v_{p})v_{p} \\
%&\leq& 2^{p/4}\int_{\Omega}|v_{p}|^{p/2} \leq 2^{p/4} (\int_{\Omega} |v_{p}|^{2^{*}})^{p/2 2^{*}}
% |\Omega|^{\frac{2 2^{*} - p}{22^{*}}} \leq 2^{p/4} S_{N}^{p/2}\|v_{p}\|_{2}^{p/2} |\Omega|^{\frac{2 2^{*} - p}{22^{*}}} 
%\end{eqnarray*}
%then
%$$\|v_{p}\|\geq \Big( \frac{1}{2^{p/4} S_{N}^{p/2} |\Omega|^{\frac{2 2^{*} - p}{22^{*}}}}\Big)^{2/(p-4)}$$
%hence
%$$\liminf_{p\to 22^{*}} \|v_{p}\| \geq 
%\Big( \frac{1}{2^{2^{*}/2} S_{N}^{2^{*}} }\Big)^{1/(2^{*}-2)}.$$
%\begin{eqnarray}\label{dis}
%\|u_{ p}\|^{2} &\le &\|u_{p}\| + \int_{\Omega} f(u_{p}) f'(u_{p})u_{p} \\
%&=&\int_{\Omega}|f(u_{ p})|^{p-2}f(u_{p}) f'(u_{p})u_{p}\\
%&\leq&\int_{\Omega}|f(u_{ p})|^{p} 
%\leq  \int_{\Omega}|u_{p}|^{p}  \leq C \|u_{p}\|^{p}
%\end{eqnarray}
%where $C$ is a positive constant which can be made independent of $p$.
%(see e.g. \cite{LL})
%Hence
%\begin{equation*}
%\exists\,c>0 \ \ \mbox{ s.t. } \forall\, p\in(4,22^{*})\, :\ \ 0< c\le\|v_{p}\|.
%\end{equation*}

We address now two limit cases related to our equation.
They involve the critical problems both in the domain $\Omega$
and in the whole space $\mathbb R^{N}$.

\subsection{Behavior of the family of ground state levels $\{\mathfrak m_{p}\}_{p\in (4,22^{*})}$}%\label{sec:limit}
We introduce the critical problem in the domain $\Omega$.
This  is done in order   to evaluate
the limit of the ground state levels  $\{\mathfrak m_p\}_{p\in(4,22^*)}$
when $p\rightarrow22^*$. The main theorem in this subsection is Theorem \ref{prop:limitem*},
which requires  first some preliminary work.

\medskip

Let us introduce the $C^{1}$ functional associated to $p=22^{*}$,
\[
I_{*}(v):=I_{22^{*}}(v)=\frac{1}{2}\int_{\Omega} |\nabla v|^{2}
%+\frac{1}{2}\int_{\Omega} |f(u)|^{2}
-\frac{1}{22^{*}}\int_{\Omega}|f(v)|^{22^{*}}, \qquad v\in H^{1}_{0}(\Omega)
\]
whose critical points are the solutions of
\begin{equation}
\left\{
\begin{array}[c]{ll}
\label{star} -\Delta v =|f(v)|^{22^{*}-2}f(v)f'(v) & \quad\text{ in } \Omega\\
v=0 & \quad\text{ on } \partial\Omega.
\end{array}
\right.
\end{equation}
It is known that the lack of compactness of the embedding of $H^{1}_{0}(\Omega)$
in $L^{2^{*}}(\Omega)$ implies that 
$I_{*}$ does not satisfies the $(PS)$ condition at every level.
This is due to the invariance with respect to the conformal scaling 
$$u(\cdot)\longmapsto v_{R}(\cdot):=R^{N/2^{*}}v(R(\cdot)) \ \ \ \ (R>1)$$
which leaves invariant the $L^2-$norm of the gradient as well as the $L^{2^*}-$norm, i.e.
$|\nabla v_{R}|_{2}^{2}=|\nabla v|_{2}^{2}$ and $|v_{R}|_{2^{*}}^{2^{*}}=|v|_{2^{*}}^{2^{*}}\,.$

Related to the critical problem we have the following:

\begin{lemma}\label{lem:nonex}
	If $\Omega$ is a star-shaped domain
	then there exists only the trivial solution to \eqref{star}.
\end{lemma}
\begin{proof}
Let $v\in H^{1}_0(\Omega)$ be a solution to \eqref{star}. Setting $h(s)=|f(s)|^{22^{*}-2}f(s)f'(s)$,  we have $H(v)\in L^1(\Omega)$ where $H(s)=\frac{1}{22^*}|f(s)|^{22^*}$. Moreover, according to Br\'ezis-Kato theorem and by elliptic regularity theory, it is easily seen that $u\in C^2(\Omega)\cap C^1(\overline{\Omega})$. Thus, by using
the  Pohozaev identity (see e.g.  \cite[Theorem B.1]{Willem}) we obtain
$$
\frac{1}{2}\int_{\partial\Omega}|\nabla v|^2\sigma.\nu d\sigma+\frac{N-2}{2}\int_{\Omega}|\nabla v|^2=\frac{N}{22^*}\int_{\Omega}|f(v)|^{22^*}
$$
where $\nu$ denotes the unit outward normal to $\partial\Omega$. Since $v$ is a solution, one also has
$$
\int_{\Omega}|\nabla v|^2=\int_{\Omega}|f(v)|^{22^*-2}f(v)f'(v)v
$$
Now, combining the last two equalities we reach
$$
\frac{1}{2}\int_{\partial\Omega}|\nabla v|^2\sigma.\nu d\sigma=\frac{N}{22^*}\int_{\Omega}|f(v)|^{22^*}-\frac{N-2}{2}\int_{\Omega}|f(v)|^{22^*-2}f(v)f'(v)v
$$
Using that $f(v)f'(v)v\geq f^2(v)/2$ it follows that $\int_{\partial\Omega}|\nabla v|^2\sigma.\nu d\sigma\leq 0$ and we must have $v=0$ provided that $\sigma.\nu>0$ on $\partial\Omega$. 
\end{proof}

Let
$$
\mathcal{N}_{*}=\left\{v\in H^{1}_{0}(\Omega)\setminus \{0\}:  G_{*}(v):=I'_{*}(v)[v]=0\right\}\,, 
% G_{*}(v)=\int_{\Omega} |\nabla v|^{2}-\int_{\Omega}|f(v)|^{22^{*}}
$$
be the Nehari manifold 
%(that fact that it is a manifold is done exactly as in Lemma \ref{lemmanehari})
associated to the critical problem \eqref{star}. By Lemma \ref{lemmanehari}
%is available also for the critical problem;
%contains all the solutions of \eqref{star} and, again as in Lemma \ref{lemmanehari}, 
it results
%we see that, for $v\in \mathcal N_{*}$ it holds
%\begin{eqnarray*}
%I_{*}(v) & = &I_{*}(v) - \frac{1 }{2^{* }} I'_{*}(v)[v] \\
%&= & \Big(\frac{1}{2} - \frac{1}{2^{*}}\Big) \int_{\Omega}|\nabla v|^{2}
%+\frac{1}{2^{*}} \int_{\Omega}|f(v)|^{22^{*} - 2}f(v) f'(v)v -\frac{1}{22^{*}} \int_{\Omega}|f(v)|^{22^{*}}\\
%&\geq& \Big(\frac{1}{2} - \frac{1}{2^{*}}\Big) \int_{\Omega}|\nabla v|^{2},
%\end{eqnarray*}
%where we have used (6) of Lemma \ref{Lema f}.
%Then 
\begin{equation}\label{eq:m*}
m_{*}:=\inf_{v\in\mathcal{N}_{*}} I_{*}(v)> 0. 
\end{equation}
In contrast to the case $p\in(4,22^{*})$, now $m_{*}$ is not achieved. 

The value $m_*$ turns out to be an upper bound for the sequence of ground states levels $\{\mathfrak m_p\}_{p\in(4,22^*)}$,
as we will prove below. First we need a lemma.

\begin{lemma}\label{le:general}
Let $w\in H^{1}_{0}(\Omega)\setminus\{0\}$ be fixed. For every
$p\in (4,22^{*})$ let $t_{p}=t_{p}(w)>0$ given in  \eqref{lemmanehariiv} of Lemma \ref{lemmanehari}, i.e. such that $t_{p} w\in \mathcal N_{p}$.
Then 
$$\lim_{p\to 22^{*}} t_{p}(w)=\overline t>0 \quad \text{ and } \quad \overline t w\in \mathcal N_{*}.$$
Moreover   if $w\in \mathcal N_{*}$ then $\lim_{p\to 22^{*}} t_{p}(w) = 1$.
\end{lemma}
\begin{proof}
%Let us set $t_{p}=t_{p}(w)$.
By definition
\begin{equation}\label{eq:generalNp}
t_{p}^{2}\int_{\Omega}|\nabla w|^{2}= \int_{\Omega}|f(t_{p}w)|^{p-2}f(t_{p}w) f'(t_{p}w)t_{p}w
\end{equation}
and then, by   \eqref{eq:util} and \eqref{Lemafv} of Lemma \ref{Lema f} we have
\begin{eqnarray*}
t_{p}^{2}\int_{\Omega}|\nabla w|^{2}
\leq \int_{\Omega}|f(t_{p}w)|^{p } \leq 2^{p/4}t_{p}^{p/2}\int_{\Omega}|w|^{p/2}.
\end{eqnarray*}
Then $$t_{p}^{(p-4)/2}\geq \frac{\|w\|^{2} }{ 2^{p/4}\displaystyle\int_{\Omega}|w|^{p/2}}$$
from which it follows $$\liminf_{p\to 2 2^{*}} t_{p} \geq \Big( \frac{ \|w\|^{2} } { 2^{2^{*}/2}|w|_{2^{*}}^{2^{*}}}   
\Big)^{1/(2^{*}-2)}>0.$$

Assume now that $t_{p}\to+\infty$ as $p\to 22^{*}$.
Using again \eqref{eq:util}, by \eqref{eq:generalNp} we infer
  $t_{p}^{2}\|w\|^{2} \geq \frac12 \int_{\Omega} |f(t_{p}w)|^{p}$ and then
\begin{eqnarray*}
\|w\|^{2}&  \geq& \frac{1}{2} \int_{\{x\in\Omega: w(x)\neq0 \}} \frac{f(t_{p}|w|)^{p}  }{t_{p}^{2} |w|^{2}} |w|^{2} \\
&= & \frac1 2 \int_{\{x\in\Omega: w(x)\neq0 \}}  \frac{ f(t_{p} |w|) ^{p } \ (t_{p} |w|)^{p/2}}{(t_{p} |w|)^{p/2} \ t_{p}^{2}|w|
^{2}} |w|^{2} \\
& = &  \frac{1}{2} \int_{\{x\in\Omega: w(x)\neq0 \}}  \Big(\frac{ f(t_{p}|w|) }{  \sqrt{t_{p}|w|}}\Big)^{p} t_{p} ^{(p-4)/2}|w|^{p/2}\\
&\geq& \frac{1}{2}C  t_{p}^{(p-4)/2}\int_{\Omega} |w|^{p/2} \\
&\to& +\infty\quad\text{ as } p\to22^{*},
\end{eqnarray*}
where in the last inequality we have used 
item \eqref{outrasvii} of Lemma \ref{lem:outras}. This contradiction implies that 
 $\{t_{p}\}_{p\in (4,22^{*})}$ has to be bounded.
Then we can assume $\lim_{p\to 22^{*}} t_{p} = \overline t >0$
and passing to the limit in \eqref{eq:generalNp}, by the Dominated Convergence Theorem we get
\begin{equation}\label{eq:tw}
\overline t^{2} \|w\|^{2}=\int_{\Omega}|f(\overline t w)|^{22^{*}-2}f(\overline t w)f'(\overline t w) \overline t w
\end{equation}
i.e. $\overline t w\in \mathcal N_{*}$, proving the first part of the Lemma.

\medskip

In the case $w\in \mathcal N_{*}$, by definition
\begin{equation*}\label{eq:wcritico}
\|w\|^{2} = \int_{\Omega}|f(w)|^{22^{*}-2} f(w)f'(w) w,
\end{equation*}
which, joint with \eqref{eq:tw} gives
$$\int_{\{x\in \Omega: w(x)\neq0 \}} \frac{|f(\overline tw)|^{22^{*}-2}f(\overline t w)f'(\overline t w)}{\overline tw } w^{2}
=
\int_{\{x\in \Omega: w(x)\neq0 \}} \frac{|f( w)|^{22^{*}-2}f(  w)f'(  w)}{ w } w^{2}.
$$
The conclusion now follows since, if $w(x)\neq0$, by item \eqref{monotonicidadeiii} of Corollary \ref{monotonicidades} the map $ f(t)^{22^{*}-1} f(t)f'(t)t^{-1}$ is increasing for $t>0$ .
\end{proof}

\begin{proposition}\label{limitate}
We have
$$\limsup_{p\rightarrow22^{*}} \mathfrak m_{p}\le m_{*}.$$
\end{proposition}
\begin{proof}
Fix $\varepsilon>0.$ By definition of $m_{*}$ there exists $\overline v\in\mathcal{N}_{*}$ such that
\begin{equation*}
\label{e/2}I_{*}(\overline v)=\frac{1}{2}\| \overline v\|^{2}-\frac{1}{22^{*}}\int_{\Omega}|f(\overline v)|^{22^{*}}<m_{*}+\varepsilon.
\end{equation*}
For every $p\in (4,22^{*})$ there exists a unique $t_{p}=t_{p}(\overline v)>0$ such that $t_{p}\overline v\in \mathcal N_{p}$
and by Lemma \ref{le:general} we know that $\lim_{p\to22^{*}}t_{p}= 1$.
%Then
%\begin{eqnarray*}
%t_{p}^{2}\int_{\Omega}|\nabla v|^{2}= \int_{\Omega}|f(t_{p}v)|^{p-2}f(t_{p}v) f'(t_{p}v)t_{p}v
%\leq \int_{\Omega}|f(t_{p}v)|^{p } \leq 2^{p/4}t_{p}^{p/2}\int_{\Omega}|v|^{p/2}
%\end{eqnarray*}
%hence $$t_{p}^{(p-4)/2}\geq \frac{\|v\|^{2} }{ 2^{p/4}\displaystyle\int_{\Omega}|v|^{p/2}}$$
%from which it follows $$\liminf_{p\to 2 2^{*}} t_{p} \geq \Big( \frac{ \|v\|^{2} } { 2^{2^{*}/2} \int_{\Omega}|v|^{2^{*}}}   
%\Big)^{1/(2^{*}-2)}>0.$$
Then
$$\mathfrak  m_{p}\leq I_{p}(t_{p} \overline v) = \frac{t_{p}^{2}}{2}\|\overline v\|^{2} - 
\frac{t_{p}^{p}}{p} \int_{\Omega}|f(t_{p}\overline v)|^{p}$$
and so
$\limsup_{p\to 22^{*}}\mathfrak  m_{p} \leq I_{*}(\overline v)<m_{*} +\varepsilon$
concluding the proof.
\end{proof}

In particular we deduce the following:
\begin{corollary}\label{cor:nozero}
The family of minimizers $\{\mathfrak  g_{p}\}_{p\in(4,22^{*})}$ is bounded in $H^{1}_{0}(\Omega)$.
\end{corollary}
\begin{proof}
Since as in  \eqref{eq:nonvanishing}
\begin{eqnarray*}
\mathfrak  m_{p}= I_{p}(\mathfrak g_{p}) - \frac{2}{p}I_{p}'(\mathfrak g_{p})[\mathfrak g_{p}]
\geq \frac{p-4}{2p}\|\mathfrak g_{p}\|^{2} ,
\end{eqnarray*}
 the conclusion follows from Proposition \ref{limitate}.
%\textcolor{red}{cancellar: and then  the family of ground state solutions $\{v_{p}\}_{p\in(4,22^{*})}$ is bounded
% in $H^{1}_{0}(\Omega)$. Of course also $\{|v_{p}|_{2^{*}}\}_{p\in(4,22^{*})}$ is bounded.}
\end{proof}

It will be useful the next:
\begin{remark}\label{rem:general}
Corollary \ref{cor:nozero} can be generalized to arbitrary functions in the Nehari manifolds $\mathcal N_{p}$,
not necessary the ground states, as long as the functionals converge.

In other words, let $p_{n}\to 22^{*}$ as $n\to +\infty$. If $\{w_{n} \}\subset H^{1}_{0}(\Omega)$ is such that $w_{n}\in \mathcal N_{p_{n}}$ for every $n$, and
$I_{p_{n}}(w_{n}) \to l\in (0,+\infty)$ as $n\to\infty$, then $\{w_{n}\}$ is bounded
in $H^{1}_{0}(\Omega)$.

Indeed, similarly to the proof of Corollary \ref{cor:nozero}, this easily follows from
$$l=I_{p_{n}}(w_{n}) - \frac{2}{p_{n}}I_{p_{n}}'(w_{n})[w_{n}] + o_{n}(1)
\geq \frac{p_{n}-4}{2p_{n}}\|w_{n}\|^{2} +o_{n}(1).
$$
\end{remark}

We need now a technical lemma about the ``projections'' of the minimizers $\mathfrak g_{p}$
on the Nehari manifold of the critical problem $\mathcal N_{*}$.
Let us first observe the following which generalizes Lemma \ref{le:general}.

\begin{remark}\label{rem:tp*}
If $\{w_{p}\}_{p\in(4,22^{*})}\subset H^{1}_{0}(\Omega)$ is such that
\begin{itemize}
\item[(a)] for every $p\in (4,22^{*}): w_{p}\in \mathcal N_{p}$, \smallskip
\item[(b)] there exist $C_{1}, C_{2}>0$ such that for every $p\in (4,22^{*}): 0< C_{1}\leq |w_{p}|_{2^{*}}$ and $\|w_{p}\| \leq C_{2}$,
\end{itemize}
then setting $t_{*}(w_{p}) := t_{22^{*}}(w_{p})>0$ such that $t_{*}(w_{p})w_{p}\in \mathcal N_{*}$
(see \eqref{lemmanehariiv} of Lemma \ref{lemmanehari}), it holds
\begin{equation}\label{eq:infsup}
0< \liminf_{p\to 22^{*}} t_{*}(w_{p})\leq \limsup_{p\to 22^{*}} t_{*}(w_{p}) <+\infty.
\end{equation}
Indeed we can argue similarly as in Lemma \ref{le:general} 
(where we had a single function $w$).
By definition and  \eqref{eq:util} and \eqref{Lemafv} of Lemma \ref{Lema f}, we have
\begin{eqnarray*}%\label{eq:limite}
{t_{*}(w_{p})}^{2}\|w_{p}\|^{2} &= &\int_{\Omega} |f(t_{*}(w_{p})w_{p})|^{22^{*} -2} f(t_{*}(w_{p})w_{p}) 
f'(t_{*}(w_{p})w_{p}) t_{*}(w_{p}) w_{p} \\
&\leq &\int_{\Omega}|f(t_{*}(w_{p})w_{p})|^{22^{*} }\nonumber \\
&\leq &2^{2^{*}/2}{t_{*}(w_{p})}^{2^{*}}\int_{\Omega}|w_{p}|^{2^{*}} \nonumber
\end{eqnarray*}
and hence $${t_{*}(w_{p})}^{2^{*}-2}\geq \frac{\|w_{p}\|^{2} }{ 2^{2^{*}/2}|w_{p}|_{2^{*}}^{2^{*}}}.$$
Since by assumption, as $p\to 22^{*}$,  $\{w_{p}\}_{p\in(4,22^{*})}$
does not tend to zero in $H^{1}_{0}(\Omega)$ 
and $\{|w_{p}|_{2^{*}}\}_{p\in(4,22^{*})}$ is bounded,
%(recall item 2. in Remark \ref{rem:nuovo})
 we infer that  $\{t_{*}(w_{p})\}_{p\in(4,22^{*})}$ is bounded away from zero 
and then $\liminf_{p\to 2 2^{*}} t_{*}(w_{p})>0$.
%$$\liminf_{p\to 2 2^{*}} t_{p} \geq \Big( \frac{ \|v_{p}\|^{2} } { 2^{2^{*}/2} |v_{p}|_{2^{*}}^{2^{*}}}   
%\Big)^{1/(2^{*}-2)}>0.$$
On the other hand  using  item \eqref{outrasvii} of Lemma \ref{lem:outras}
\begin{eqnarray*}
C_{2}\geq\|w_{p}\|^{2}&  \geq& \frac{1}{2} \int_{\{x\in\Omega: w_{p}(x)\neq0 \}} \frac{f(t_{*}(w_{p})|w_{p}|) ^{22^{*}} }{{t_{*}(w_{p})}^{2} |w_{p}|^{2}} |w_{p}|^{2} \\
&= & \frac1 2 
\int_{ \{x\in\Omega: w_{p}(x)\neq0 \} }  \frac{ f(t_{*}(w_{p}) |w_{p}|)^{22^{*} }  \ (t_{*}(w_{p}) |w_{p}|)^{2^{*}}}{t_{*}(w_{p})|w_{p}|)^{2^{*}} \ {t_{*}(w_{p})}^{2}|
w_{p}|^{2}} |w_{p}|^{2} \\
& = &  \frac{1}{2} \int_{ \{x\in\Omega: w_{p}(x)\neq0 \} } \Big(\frac{ f(t_{*}(w_{p})|w_{p}|) }{  \sqrt{t_{*}(w_{p})|w_{p}|}}\Big)^{22^{*}} {t_{*}(w_{p})} ^{2^{*}-2}|w_{p}|^{2^{*}}\\
&\geq& \frac{1}{2}C  {t_{*}(w_{p})}^{2^{*}-2}\int_{\Omega} |w_{p}|^{2^{*}}\\
&\geq& C' {t_{*}(w_{p})}^{2^{*}-2}
\end{eqnarray*}
and this give  that $\{t_{*}(w_{p})\}_{p\in (4,22^{*})}$ has also to be bounded above when $p\to 22^{*}$,
proving \eqref{eq:infsup}.

\medskip

%
%
%in view of Remark \ref{rem:nuovo} and Corollary \ref{cor:nozero}  
%inequalities in \eqref{eq:infsup} are true for $t^{*}_{p}(\mathfrak g_{p})$.
\end{remark}

%Actually \eqref{eq:infsup} can be improved if $w_{p}=\mathfrak g_{p}$. Indeed we have:
% that is for the  family of minimizers $\{\mathfrak g_{p}\}_{p\in(4,22^{*})}.$

\begin{proposition}\label{prop:importante}
Assume that $\{w_{p}\}_{p\in(4,22^{*})}\subset H^{1}_{0}(\Omega)$ is such that
\begin{itemize}
\item[(a)] for every $p\in (4,22^{*}): w_{p}\in \mathcal N_{p}$, \smallskip
\item[(b)] there exist $C_{1}, C_{2}>0$ such that 
 $$\forall p\in (4,22^{*}): 0< C_{1}\leq |w_{p}|_{2^{*}}\  \text{ and }\ \|w_{p}\| \leq C_{2},$$
\item[(c)] $w_{p}\geq0$ for every $p\in (4,22^{*})$.
\end{itemize}
 Let $t_{*}(w_{p})>0$
the unique value such that
$t_{*}(w_{p})w_{p}\in\mathcal{N}_{*}$. Then
$$\lim_{p\rightarrow 2 2^{*}} t_{*}(w_{p})=1.$$
In particular
$$\lim_{p\rightarrow 2 2^{*}} t_{*}(\mathfrak g_{p})=1.$$

\end{proposition}
\begin{proof}
We assume that $p_{n}\to 22^{*}$ as $n\to+\infty$ and $w_{n}:=w_{p_{n}}\in \mathcal N_{p_{n}}$.
In virtue of (c) it is $f(w_{n}), f'(w_{n})\geq0$.
Moreover by Remark \ref{rem:tp*} we can assume that 
$$\lim_{n\to+\infty} t_{*}(w_{n}) = t_{0}>0.$$

Let us begin by proving the following:

{\bf Claim:} $A_{n}:=\displaystyle \Big|\int_{\Omega} f(w_{n})^{22^{*}-2} f(w_{n}) f'(w_{n})w_{n} 
-\int_{\Omega} f(w_{n})^{p_{n}-2}f(w_{n})f'(w_{n})w_{n}\Big| = o_{n}(1)$.

Indeed let us fix $\gamma\in(0,1)$ and let $n_{0}\in \mathbb N$ such that $22^{*}-p_{n_{0}}<\gamma$.
We have
\begin{eqnarray*}
A_{n}&\leq& \int_{\Omega} | f(w_{n})^{22^{*}-2} -f(w_{n})^{p_{n}-2}| f(w_{n})f'(w_{n})w_{n}\\
&=& \int_{\{0< f(w_{n})<\gamma\}} | f(w_{n})^{22^{*}-2} -f(w_{n})^{p_{n}-2}| f(w_{n})f'(w_{n})w_{n}\\  
&+&  \int_{\{\gamma\leq f(w_{n})\leq1\}} | f(w_{n})^{22^{*}-2} -f(w_{n})^{p_{n}-2}| f(w_{n})f'(w_{n})w_{n}\\
 &+&\int_{\{ f(w_{n})>1\}} | f(w_{n})^{22^{*}-2} -f(w_{n})^{p_{n}-2}| f(w_{n})f'(w_{n})w_{n} \\ 
 &=:& a_{n}+b_{n}+c_{n}.
\end{eqnarray*}
Let us estimate $a_{n}, b_{n},c_{n}$.

Using \eqref{eq:util}, and being 
 $p_{n}$ and $22^{*}$ greater then $4$, we have
\begin{eqnarray*}%\label{eq:}
a_{n}&\leq& \int_{\{0< f(w_{n})<\gamma\}} f(w_{n})^{22^{*}} + \int_{\{0< f(w_{n})<\gamma\}} f(w_{n})^{p_{n}}\\
&\leq& \gamma ^{22^{*}}|\Omega| + \gamma^{p_{n}}|\Omega|\\
&\leq &2\gamma^{4}|\Omega|.
\end{eqnarray*}

By the Mean Value Theorem, for some $\xi_{n}\in (p_{n}-2, 22^{*}-2)$ it is (again by \eqref{eq:util})
\begin{eqnarray*}
b_{n} &=& \int_{\{\gamma\leq f(w_{n}) \leq1\}} f(w_{n})^{\xi_{n}} (p_{n} - 22^{*}) \ln (f(w_{n})) f(w_{n})f'(w_{n})w_{n}\\
&\leq& (p_{n} - 22^{*}) \ln (\gamma) \int_{\{ \gamma\leq f(w_{n})\leq 1\}} f(w_{n})^{\xi_{n}+2}\\
&\leq& (p_{n} - 22^{*}) |\Omega| \ln (\gamma)  \\ 
&=&o_{n}(1).
\end{eqnarray*}

Finally,
\begin{eqnarray*}
c_{n} &=& \int_{\{ f(w_{n})>1\}} \Big( f(w_{n})^{22^{*}-2} - f(w_{n})^{p_{n}-2}\Big)f(w_{n})f'(w_{n})w_{n} \\
&\leq&  \int_{\{ f(w_{n})>1\}} \Big( f(w_{n})^{22^{*}} - f(w_{n})^{p_{n}}\Big) \\
&\leq& \int_{\{ f(w_{n})>1\}} \Big( f(w_{n})^{22^{*}} - f(w_{n})^{p_{n_{0}}}\Big)
\end{eqnarray*}
where we used that $f(w_{n})^{p_{n}}>f(w_{n})^{p_{n_{0}}}$ for $n>n_{0}$.
Using again the Mean Value Theorem, for some $\xi_{0} \in (p_{n_{0}}, 22^{*})$,
and the fact that for $s>1$ it is $\ln (s)\leq C s^{22^{*} -\xi_{0}}$,
we have for every $n>n_{0}$
\begin{eqnarray*}
 c_{n}&\leq&\int_{\{ f(w_{n})>1\}} f(w_{n})^{\xi_{0}} \ln(f(w_{n})) (22^{*} -p_{n_{0}}) \\
 &\leq&(22^{*} - p_{n_{0}}) C \int_{\{ f(w_{n})>1\}} f(w_{n})^{22^{*}}\\
 &\leq& C \gamma,
\end{eqnarray*}
where we used that  $\int_{\{ f(w_{n})>1\}} f(w_{n})^{22^{*}}\leq C$. 

\medskip

Summing up, $A_{n} \leq 2\gamma^{4}|\Omega| + o_{n}(1) + C\gamma$ and then
$$\limsup_{n\to+\infty}A_{n}\leq 2\gamma^{4}|\Omega|  + C\gamma$$
which  proves  Claim.

\medskip

Observe now that, by \eqref{eq:util},
$$
0\leq \int_{\Omega} f(w_{n})^{22^{*}-2}f(w_{n})f'(w_{n})w_{n} 
\leq \int_{\Omega}f(w_{n})^{22^{*}}
\leq C\int_{\Omega} w_{n}^{2^{*}} \leq C',
$$
then up to subsequences  we have
\begin{equation}\label{eq:claim0L}
\int_{\Omega} f(w_{n})^{22^{*}-2}f(w_{n})f'(w_{n})w_{n} \to L
%\geq \frac12\int_{\Omega}f^{22^{*}}(w)
\end{equation}
and then the above Claim  gives
\begin{equation}\label{eq:claim1L}
\|w_{n}\|^{2}=\int_{\Omega} f(w_{n})^{p_{n}-2}f(w_{n})f'(w_{n})w_{n} \to L\geq0.
\end{equation}
Since $w_{n}\not\to0$ in $H^{1}_{0}(\Omega)$ (being the Nehari manifolds 
uniformly bounded away from zero, see Remark \ref{rem:limbaixo}), it has to be  $L>0$.

\medskip

 Suppose that $t_{0}>1$; hence, for large $n$, we have $t_{*}(w_{n})>1$. Since $w_{n}\in \mathcal N_{p_{n}}$ and 
 $t_{*}(w_{n})w_{n}\in \mathcal N_{*}$, it is,
by \eqref{eq:claim1L} and  \eqref{outrasiii} and \eqref{outrasvii} of Lemma \ref{lem:outras},
\begin{eqnarray*}
\|w_{n}\|^{2} &=& \int_{\Omega} f(w_{n})^{p_{n}-2}f(w_{n})f'(w_{n})w_{n}\to L, \\
t_{*}(w_{n})^{2}\|w_{n}\|^{2} &=&\int_{\Omega} f(t_{*}(w_{n})w_{n})^{22^{*}-2}f(t_{*}(w_{n})w_{n})f'(t_{*}(w_{n})w_{n})t_{*}(w_{n})w_{n}\\
&\geq&  t_{*}(w_{n})^{2^{*}} \int_{\Omega} f(w_{n})^{22^{*}-2}f(w_{n})f'(w_{n})w_{n}.
\end{eqnarray*}
Passing to the limit above and using \eqref{eq:claim0L}, we deduce
$$t_{0}^{2} L \geq t_{0}^{2^{*}}L$$
and then $t_{0}\leq1$ which is a contradiction.

On the other hand, if  $t_{0}<1$, we can assume that $t_{*}(w_{n})<1$. Then as before,
\begin{eqnarray*}
\|w_{n}\|^{2} &=& \int_{\Omega} f(w_{n})^{p_{n}-2}f(w_{n})f'(w_{n})w_{n}\to L, \\
t_{*}(w_{n})^{2}\|w_{n}\|^{2} &=&\int_{\Omega} f(t_{*}(w_{n})w_{n})^{22^{*}-2}f(t_{*}(w_{n})w_{n})
f'(t_{*}(w_{n})w_{n})t_{*}(w_{n})w_{n}\\
&\leq& t_{*}(w_{n})^{2^{*}} \int_{\Omega} f(w_{n})^{22^{*}-2}f(w_{n})f'(w_{n})w_{n}
\end{eqnarray*}
and passing to the limit, by \eqref{eq:claim0L},  we deduce
$$t_{0}^{2}L\leq t_{0}^{2^{*}}L$$
and then $t_{0}\geq1$ which is again a contradiction.

\medskip

Finally,
since
$\{|\mathfrak g_{p}|_{2^{*}}\}_{p\in[4+\eta,22^{*})}$
is bounded away from zero by the final part in Remark \ref{rem:nuovo}, and
$\{\mathfrak g_{p}\}_{p\in(4,22^{*})}$ are bounded in $H^{1}_{0}(\Omega)$ by Corollary \ref{cor:nozero},
we have that
$\{\mathfrak g_{p}\}_{p\in [4+\eta,22^{*})}$ satisfy (b), and clearly (a) and (c). Then
the conclusion follows.
\end{proof}

Thanks to the previous result we get the next:

\begin{proposition}\label{le:liminf}
We have 
$$m_{*}\leq \liminf_{p\to 22^{*}} \mathfrak m_{p}  .$$
\end{proposition}
\begin{proof}
For $p_{n} \to 22^{*}$,  by Corollary \ref{cor:nozero}
we have $\mathfrak g_{n}:=\mathfrak g_{p_{n}}\rightharpoonup v$ in $H^{1}_{0}(\Omega)$
and consequently
$$\int_{\Omega}|f(\mathfrak g_{n})|^{p_{n}} \to \int_{\Omega} |f(v)|^{22^{*}}\,, \qquad 
\int_{\Omega} |f(t_{*}(\mathfrak g_{n})\mathfrak g_{n})|^{22^{*}} \to \int_{\Omega}|f(v)|^{22^{*}}.$$
Furthermore, by  Proposition \ref{prop:importante} we have $t_{*}(\mathfrak g_{n})\to 1$.
Since by definition
$$\frac{1}{2}\|\mathfrak g_{n}\|^{2} = \mathfrak  m_{p_{n}}+\frac{1}{p_{n}}\int_{\Omega}|f(\mathfrak g_{n})|^{p_{n}},$$
we get
\begin{eqnarray*}
m_{*}&\leq& I_{*}(t_{*}(\mathfrak g_{n})\mathfrak g_{n}) \\
&=& \frac{{t_{*}(\mathfrak g_{n})}^{2}}{2}\|\mathfrak g_{n}\|^{2} -\frac{1}{22^{*}}\int_{\Omega}
|f(t_{*}(\mathfrak g_{n})\mathfrak g_{n})|
^{22^{*}}\\
&=& {t_{*}(\mathfrak g_{n})}^{2} \mathfrak m_{p_{n}}+\frac{{t_{*}(\mathfrak g_{n})}^{2}}{p_{n}}\int_{\Omega}|f(\mathfrak g_{n})|^{p_{n}}
-\frac{1}{22^{*}}\int_{\Omega}|f(t_{*}(\mathfrak g_{n})\mathfrak g_{n})|^{22^{*}}
\end{eqnarray*}
and passing to the limit  we deduce $m_{*} \leq \liminf_{n\to+\infty}\mathfrak m_{p_{n}}$.
\end{proof}

By Proposition \ref{limitate} and Proposition \ref{le:liminf} we deduce the  desired result.
\begin{theorem}\label{prop:limitem*}
For any bounded domain $\Omega$, it holds
$$\lim_{p\to 22^*} \mathfrak m_{p} =m_{*}.$$
\end{theorem}

\subsection{A local Palais-Smale condition for $I_{*}$}\label{subsec:PS}

Let us recall the following Brezis-Lieb type splitting involving the function $f$
available when $w_{n}:=v_{n}-v\rightharpoonup0$ in $H^{1}_{0}(\Omega)$.
%\begin{lemma}{(\cite[Lemma 3.1.9]{JF})}
%\textcolor{red}{NAO SE USA ESSE LEMMA??}
%If $\{v_{n}\}$ is a sequence of measurable functions defined in $\Omega$ and such that
%\begin{itemize}
%\item[1] $v_{n} \to v$ a.e. in $\Omega$, \medskip
%\item[2.] there exists $C> 0$  such that  $\int_{\Omega} |f(v_{n})|^{2} \leq C$ for any $n$,
%\end{itemize}
%then
%$$\int_{\Omega} |f(v_{n} - v)|^{2} = \int_{\Omega} |f(v_{n})|^{2} - \int_{\Omega}|f(v)|^{2} + o_{n}(1).$$
%\end{lemma}
%\begin{remark}\label{rem:BL}
They will be useful in the next two lemmas.

The first
one concerns the power nonlinearity, that is  
\begin{equation}\label{eq:1BL}
\int_{\Omega} |f(v_{n})|^{22^{*}} = \int_{\Omega}|f(v)|^{22^{*}} +\int_{\Omega}|f(w_{n})|^{22^{*}}+o(1).
\end{equation}
See  \cite[equation (3.11)]{AUG} and observe that 
the splitting also holds  in the critical case $p=22^{*}$.

The second one  is 
$$\int_{\mathbb R^{N}}  \Big| |f(w_{n})|^{22^{*}-2} f(w_{n})f'(w_{n}) w_{n}
-|f(v_{n})|^{22^{*}-2} f(v_{n})f'(v_{n}) v_{n} + |f(v)|^{22^{*}-2} f(v)f'(v) v
\Big|^{\alpha}= o_{n}(1)$$
which holds for  some $\alpha\in (2, 2^{*})$, 
see \cite[equation (3.14)]{AUG}.
However in a bounded domain we can even allow $\alpha=1$
and consequently we obtain
\begin{eqnarray}
\int_{\Omega}  |f(w_{n})|^{22^{*}-2} f(w_{n})f'(w_{n}) w_{n} &=&
\int_{\Omega}|f(v_{n})|^{22^{*}-2} f(v_{n})f'(v_{n}) v_{n} \label{eq:2BL}\\
&-& |f(v)|^{22^{*}-2} f(v)f'(v) v+ o_{n}(1).\nonumber 
\end{eqnarray}
%\end{remark}

\medskip

To show the local  Palais-Smale condition for $I_{*}$
it will be useful the next auxiliary result.
Recall that $S$ is the best Sobolev constant of the embedding $H^{1}_{0}(\Omega)$ into $L^{2^{*}}(\Omega)$.

\begin{lemma}\label{lem:limitacaoPS} 
Let $\{v_{n}\}$ be a  $(PS)$ sequence for the functional $I_{*}$ at level $d\in \mathbb R$.
Then, up to subsequences
\begin{itemize}
\item[1.] $v_{n}\rightharpoonup v$ in $H^{1}_{0}(\Omega)$,
\item[2.] $I'_{*}(v) = 0$, i.e. $v$ is a solution of \eqref{star},
\item[3.] setting, $w_{n}:=v_{n} - v$, then
$$I_{*}(v_{n}) = I_{*}(v) +I_{*}(w_{n}) +o_{n}(1) \quad \text{ and }
\quad I'_{*}(w_{n}) \to 0.$$
In particular $\{w_{n}\}$ is  a $(PS)$ sequence for $I_{*}$ at level $d-I_{*}(v)$.
\end{itemize} 
\end{lemma}
\begin{proof}
If $d\in \mathbb R,$  $I_{*}(v_{n}) \to d$ and $I_{*}'(v_{n})\to 0$ then
\begin{equation*}\label{eq:}
I_{*}(v_{n})- \frac{1}{2^{*}} I_{*}'(v_{n})[v_{n}] \leq  C(1 +\|v_{n}\|). 
\end{equation*}
On the other hand, by the computation above
% by (6) of Lemma \ref{Lema f}
$$
I_{*}(v_{n}) - \frac{1 }{2^{* }} I'_{*}(v_{n})[v_{n}] 
%&= & \Big(\frac{1}{2} - \frac{1}{2^{*}}\Big) \int_{\Omega}|\nabla v_{n}|^{2}
%+\frac{1}{2^{*}} \int_{\Omega}|f(v_{n})|^{22^{*} - 2}f(v_{n}) f'(v_{n})v_{n} -\frac{1}{22^{*}} \int_{\Omega}|f(v_{n})|^{22^{*}}\\
\geq \frac{2^{*}-2}{22^{*}} \int_{\Omega}|\nabla v_{n}|^{2}
$$
and the boundedness of $\{v_{n}\}$ follows.
Then we can assume  that $v_{n} \rightharpoonup v$ in $H^{1}_{0}(\Omega)$
with strong convergence  $L^{s}(\Omega), s\in [1,2^{*})$ and $v_{n} \to v$ a.e. in $\Omega$.
Note now, using 
\eqref{Lemafix} of Lemma \ref{Lema f},
that
$$\int_{\Omega}\Big ||f(v_{n})|^{22^{*}-2} f(v_{n})f'(v_{n}) \Big|^{2N/(N+2)} \leq C\int_{\Omega}|v_{n}|^{2N(2^{*} -1)/(N+2)}
=\int_{\Omega} |v_{n}|^{2^{*}} \leq C'.$$
Then there exists some $w\in L^{2N/(N+2)}(\Omega)$ such that, up to subsequence,
$$|f(v_{n})|^{22^{*} - 2} f(v_{n}) f'(v_{n}) \rightharpoonup  w
 \quad \text{ in } \quad L^{2N/(N+2)}(\Omega).$$
 But it is easy to see, due to the unicity of the weak limit,  that
$$|f(v_{n})|^{22^{*} - 2} f(v_{n}) f'(v_{n}) \rightharpoonup   |f(v)|^{22^{*} - 2}f(v) 
f'(v) \quad \text{ in } \quad L^{2N/(N+2)}(\Omega)$$
(note that $2N/(N+2) = (2^{*})'$).
This allows to conclude that, for every $\varphi\in H^{1}_{0}(\Omega), I_{*}'(v_{n})[\varphi]  \to I_{*}(v)[\varphi]$ and then,
since $I'_{*}(v_{n})\to 0$, we conclude that $v$ is a critical point of $I_{*}$.

\medskip

Now, by the  Brezis-Lieb  splitting  \eqref{eq:1BL},
%in Remark \ref{rem:BL}
 we have
\begin{eqnarray*}\label{eq:convergenza}
I_{*}(v_{n}) &=&\frac{1}{2}\int_{\Omega}|\nabla w_{n} + \nabla v|^{2} -\frac{1}{22^{*}}\int_{\Omega}|f(w_{n}+v)|^{22^{*}} \nonumber\\
&=& \frac{1}{2}\int_{\Omega}|\nabla w_{n}|^{2} +\int_{\Omega}\nabla w_{n} \nabla v +\frac{1}{2}\int_{\Omega}|\nabla v|^{2} -\frac{1}{22^{*}}\int_{\Omega}|f(v)|^{22^{*}}
-\frac{1}{22^{*}} \int_{\Omega}|f(w_{n})|^{22^{*}} +o_{n}(1) \nonumber \\
&=&I_{*}(v) +I_{*}(w_{n}) +o_{n}(1).
%\frac{1}{2}\int_{\Omega}|\nabla w_{n}|^{2}-\frac{1}{22^{*}} \int_{\Omega}|f(w_{n})|^{22^{*}} +o_{n}(1)
%\longrightarrow& d.
\end{eqnarray*}
Moreover since $w_{n}\rightharpoonup 0$ in $H^{1}_{0}(\Omega)$
and  
$$|f(w_{n})|^{22^{*} - 2} f(w_{n}) f'(w_{n}) \rightarrow 0 \quad \mbox{ in } \ H^{-1}(\Omega)$$
we deduce that
$$\|I_{*}'(w_{n})\| = 
\sup_{\|\varphi\|=1} \Big| \int_{\Omega} \nabla w_{n} \nabla \varphi - \int_{\Omega}|f(w_{n})|^{22^{*}-2} f(w_{n})f'(w_{n})\varphi\Big| \rightarrow 0$$
concluding the proof.
\end{proof}

Then we have the local $(PS)$ condition for the functional $I_{*}$.

\begin{proposition}\label{lem:PS}
The functional $I_{*}$ satisfies the $(PS)$ condition at level $d\in \mathbb R$, for $$d<\frac{1}{N}\Big(\frac{S}{2}\Big)^{N/2}.$$
\end{proposition}
\begin{proof}
Let $\{v_{n}\}$ be a $(PS)_{d}$ sequence for $I_{*}$. We know that $v_{n}\rightharpoonup v$
in $H^{1}_{0}(\Omega) , I'_{*}(v)=0$ and $I_{*}(v)\geq0$.

By defining  $w_{n}: =v_{n} -v $, and using that $\int_{\Omega}|f(w_{n})|^{22^{*}} \leq C \int_{\Omega}|\nabla w_{n}|^{2}$ we  have
\begin{equation}\label{eq:bhatb}
\int_{\Omega}|\nabla w_{n}|^{2} \to A\geq0, \qquad \int_{\Omega} |f(w_{n})|^{22^{*}}\to  B\geq0.
\end{equation}
All that we need to show is that $A=0$.

By using  the Brezis-Lieb splitting \eqref{eq:2BL}
% see Remark \ref{rem:BL},
 we have
\begin{equation*}
I_{*}'(w_{n})[w_{n}] = I_{*}'(v_{n})[v_{n}] - I'_{*}(v)[v]+ o_{n}(1) = o_{n}(1),
\end{equation*}
which explicitly is
\begin{equation}\label{eq:convergencia}
\int_{\Omega}|\nabla w_{n}|^{2} - \int_{\Omega}|f(w_{n})|^{22^{*} -2 }f(w_{n})f'(w_{n})w_{n} =o_{n}(1).
\end{equation}
So, in virtue of \eqref{eq:convergencia}, we deduce
%\begin{multline}\label{eq:naosei}
%\int_{\Omega}|\nabla w_{n}|^{2} -\int_{\Omega}|f(w_{n})|^{22^{*} -2} f(w_{n})f'(w_{n})w_{n}  \\
%= \int_{\Omega}|\nabla v_{n}|^{2} -\int_{\Omega}|\nabla v|^{2} - \int_{\Omega}|f(v_{n})|^{22^{*}} f(v_{n})f'(v_{n})v_{n}+\int_{\Omega}|f(v)|^{22^{*}-2}f(v)f'(v)v +o_{n}(1) \\
%=I'_{*}(v_{n})[v_{n}]
%\end{multline}
%and using that
%$$\int_{\Omega}|\nabla v_{n}|^{2} -\int_{\Omega}|f(v_{n})|^{22^{*} -2} f(v_{n})f'(v_{n})v_{n}=I'_{*}(v_{n})[v_{n}]\to 0,$$
%it follows
%\begin{eqnarray}\label{eq:conv}
%\int_{\Omega}|\nabla w_{n}|^{2} -\int_{\Omega}|f(w_{n})|^{22^{*} -2}f(w_{n})f'(w_{n})w_{n}
%&\to& \int_{\Omega}|f(v)|^{22^{*}-2}f(v)f'(v)v-\textcolor{red}{+} \int_{\Omega}|\nabla v|^{2} \nonumber\\
%&=& I'_{*}(v)[v]=0
%\end{eqnarray}
\begin{equation}\label{eq:deduce}
\int_{\Omega}|f(w_{n})|^{22^{*} -2 }f(w_{n})f'(w_{n})w_{n} \longrightarrow A.
\end{equation}
Then 3. of Lemma \ref{lem:limitacaoPS} and   \eqref{eq:convergenza} imply
\begin{equation}\label{eq:Ibd}
d=I_{*}(v)+\frac{A}{2}-\frac{B}{22^{*}}\geq \frac{A}{2}-\frac{B}{22^{*}}.
\end{equation}
By \eqref{Lemafvi} of Lemma \ref{Lema f} it holds 
$$\frac{1}{2}\int_{\Omega}|f(w_{n})|^{22^{*}} \leq \int_{\Omega}|f(w_{n})|^{22^{*}-2}
f(w_{n})f'(w_{n})w_{n} \leq \int_{\Omega}|f(w_{n})|^{22*} $$
so that, by \eqref{eq:bhatb} and \eqref{eq:deduce},
\begin{equation}\label{eq:hatb}
 \frac{1}{2} B \leq  A \leq B.
\end{equation}
Then, coming back to \eqref{eq:Ibd} we infer
\begin{eqnarray}\label{eq:b*}
\frac{1}{N} A =\frac{A}{2} - \frac{A}{2^{*}} \leq d.
\end{eqnarray}
%The Sobolev inequality applied to $f^{2}(w_{n})$, we get
%\begin{equation*}
%S \Big(\int_{\Omega} |f(w_{n})|^{22^{*}}\Big)^{2/2^{*}} \leq \int_{\Omega}|\nabla f^{2}(w_{n})|^{2}.
%\end{equation*}
Now, by the Sobolev inequality applied to $f(w_{n})^{2}$ and  \eqref{Lemafix} of Lemma \ref{Lema f} we get
\begin{equation*}
S \Big(\int_{\Omega} |f(w_{n})|^{22^{*}}\Big)^{2/2^{*}}\leq\int_{\Omega}|\nabla f(w_{n})^{2}|^{2} = \int_{\Omega}|2f(w_{n}) f'(w_{n}) \nabla w_{n}|^{2} \leq2 \int_{\Omega}|\nabla w_{n}|^{2}
\end{equation*}
and then, passing to the limit and making use of \eqref{eq:hatb}, we arrive at
%from which it follows
%\begin{equation*}\label{eq:}
%\int_{\Omega}|\nabla w_{n}|^{2} \geq \frac{S}{2} \Big(\int_{\Omega} |f(w_{n})|^{22^{*}}\Big)^{2/2^{*}}
%\end{equation*}
$$  \frac{S}{2}A^{2/2^{*}} \leq \frac{S}{2}{B }^{2/2^{*}}  \leq A. $$
If it were $A>0$, then we deduce
\begin{equation*}\label{eq:bS}
\Big(\frac{S}{2}\Big)^{N/2}\leq A.
\end{equation*}
%Since $I_{*}(v)\geq 0$ and $\hat b\leq 2b$, by \eqref{eq:Ibd} it follows that
%$b/2 - b/2^{*} \leq d$ that is
%\begin{equation*}\label{eq:}
%\frac{1}{N} b \leq d.
%\end{equation*}
But then using \eqref{eq:b*} we get
\begin{equation*}\label{eq:d}
\frac{1}{N} \Big( \frac{S}{2}\Big)^{N/2}\leq \frac{1}{N}A \leq d < \frac{1}{N} \Big( \frac{S}{2}\Big)^{N/2}
\end{equation*}
and this contradiction implies that $A=0$, concluding the proof.
\end{proof}

\subsection{A global compactness result}
In  order to prove our multiplicity results we need to deal 
with another ``limit'' functional, now related to the 
critical problem in the whole $\mathbb R^{N}$.

Let us  introduce the space $D^{1,2}(\mathbb{R}^{N})=\left\{u\in L^{2^{*}}
(\mathbb{R}^{N}): |\nabla u|\in L^{2}(\mathbb R^{N})\right\}$ which can also be characterized as the
closure of $C^{\infty}_{0}(\mathbb{R}^{N})$ with respect to the (squared)
norm
$$\|u\|^{2}_{D^{1,2}(\mathbb{R}^{N})}=\int_{\mathbb{R}^{N}} |\nabla u|^{2}.$$
A function in  $H^{1}_{0}(\Omega)$ can be thought as an element of $D^{1,2}(\mathbb{R}^{N})$.

Let us define the functional 
\begin{equation*}\label{eq:IRN}
\widehat I(v) = \frac{1}{2}\int_{\mathbb R^{N}}|\nabla v|^{2} - \frac{1}{22^{*}} \int_{\mathbb R^{N}} |f(v)|^{22^{*}}
\end{equation*}
whose critical points are the weak solutions of
\begin{equation}\label{eq:problRN}
\left\{
\begin{array}[c]{ll}
-\Delta u =|f(v)|^{22^{*}-2}f(v)f'(v)\medskip  \\ 
v\in D^{1,2}(\mathbb R^{N}).
\end{array}
\right.
\end{equation}
%\begin{remark}\label{rem:m>0}
Setting as usual
$$\widehat{\mathcal N} = \left\{ v\in D^{1,2}(\mathbb R^{N})\setminus\{0\} : \widehat{I}' (v)[v] = 0\right\},$$
all the solutions of \eqref{eq:problRN} are in $\widehat{\mathcal N}$; it is a differentiable manifold,
is bounded away from zero,
%(the proof is exactly as in item \eqref{lemmaneharii}  of Lemma  \ref{lemmanehari});  for $v\in \widehat{\mathcal N}$,
%as in the proof of \eqref{lemmanehariii} of Lemma \ref{lemmanehari},
%$$
%\| v\|^{2} = \int_{\mathbb R^{N}} |f(v)|^{22^{*}-2} f(v)f'(v)v \leq 2^{2^{*}/2}\int_{\mathbb R^{N}}|v|^{2^{*}}\leq C\|v\|^{p/2}\\
%$$
%so that $\widehat{\mathcal N}$ is bounded away from zero 
and
$$ \widehat{m}: = \inf_{v\in \widehat{\mathcal N}} \widehat I(v)>0.$$
The proof of these facts is exactly as in \eqref{lemmaneharii}-\eqref{lemmanehariiii} of Lemma  \ref{lemmanehari}.

\medskip

As a matter of notation, in the rest of the paper given 
a function $z\in D^{1,2}(\mathbb R^{N}), \xi\in \mathbb R^{N}$
and $R>0$, we define the {\sl conformal rescaling}  $z_{R,\xi}$ as
\begin{equation}\label{eq:conformal}
z_{R, \xi} (x):= R^{N/2^{*}} z(R(x-\xi)).
\end{equation}
Of course $\|z\|_{D^{1,2}(\mathbb R^{N})}=\|z_{R,\xi}\|_{D^{1,2}(\mathbb R^{N})}.$

We need the following important Lemma
whose prove is omitted since it is like in \cite[Lemma 3.2]{Str}.
Note that the conclusion of item $(e)$
simply follows by Proposition \ref{lem:PS}.
\begin{lemma}\label{lem:preliminare}
Let $\{w_{n}\}$ be  a $(PS)_{\beta}$ sequence for $I_{*}$ 
such that $w_{n}\rightharpoonup 0$ in $H^{1}_{0}(\Omega)$.
Then there exist sequences $\{x_{n}\}\subset \Omega, \{R_{n}\}\subset (0,+\infty)$ with $R_{n}\to +\infty$, 
and a nontrivial solution $\widehat v$ of \eqref{eq:problRN} 
such that, up to subsequences,
\begin{itemize}
\item[(a)] $\widehat w_{n}: =w_{n} -\widehat v_{R_{n} , x_{n}} +o_{n}(1)$
% R_{m}^{N/2^{*}} \hat v_{0}(R_{n}(\cdot - x_{n})) +o_{n}(1)$
is a $(PS)$ sequence for $I_{*}$ in $H^{1}_{0}(\Omega)$, \medskip
\item[(b)] $\widehat w_{n} \rightharpoonup 0 $ in $H^{1}_{0}(\Omega)$, \medskip
\item[(c)] $I_{*}(\widehat w_{n}) = I_{*}(w_{n})  -\widehat I(\widehat v) +o_{n}(1)$, \medskip
\item[(d)]$R_{n}d(x_{n}, \partial\Omega) \to+\infty $, \medskip
\item[(e)] if $\beta<\beta^{*} := \frac{1}{N}\Big(\frac{S}{2}\Big)^{N/2}$ then $\{w_{n}\}$ is relatively compact;
in particular $w_{n} \to 0$ in $H^{1}_{0}(\Omega)$ and $I_{*}(w_{n}) \to \beta= 0$.
\end{itemize}
\end{lemma}

Now we can prove the following ``splitting lemma'', which is   useful
to study the behaviour of the $(PS)$ sequences
for the limit functional $I_{*}$ related to the critical problem in the domain $\Omega$.

In particular it says that, if the $(PS)$ sequences does not converges 
strongly to their weak limit, then this is due to the solutions of the  problem
in the whole $\mathbb R^{N}.$

\begin{lemma}[Splitting]\label{lem:splitting} 
Let $\{v_{n}\}\subset H^{1}_{0}(\Omega)$ be a $(PS)$ sequence for the functional
$I_{*}$. 
 Then either $\{v_{n}\}$ is convergent in $H^{1}_{0}(\Omega)$, or 
 there exist 
\begin{itemize}
\item[i.] a solution $v_{0}\in H^{1}_{0}(\Omega)\subset D^{1,2}(\mathbb R^{N})$ of problem \eqref{star}, \smallskip
\item[ii.] a number $k\in \mathbb N, k$ sequences of points $\{x^{j}_{n}\}\subset \Omega$
and $k$ sequences of radii $\{R_{n}^{j}\}$ with $R^{j}_{n}\to +\infty$, where $j=1,\ldots, k$, \smallskip
\item[iii.] nontrivial solutions $\{v^{j}\}_{j=1,\ldots, k}\subset D^{1,2}(\mathbb R^{N})$ of problem \eqref{eq:problRN} 
\end{itemize} 
such that, up to subsequences,
\begin{eqnarray}\label{eq:}
&&v_{n} - v_{0} =\sum_{j=1}^{k} v_{R_{n}^{j}, x_{n}^{j}}^{j} +o_{n}(1) \quad \text{ in }\quad D^{1,2}(\mathbb R^{N}) 
\label{eq:approxv0}\\
&&I_{*}(v_{n}) = I_{*}(v_{0})+\sum_{j=1}^{k} \widehat I(v^{j})+o_{n}(1) \label{eq:approxI}.
\end{eqnarray}
%Here  $v^{j}_{R_{n}^{j}, x_{n}^{j}}$ denotes the rescaled function
%$$v^{j}_{R_{n}^{j}, x_{n}^{j}} (x)= (R_{n}^{j})^{N/2^{*}} v^{j} (R_{n}^{j}(x-x_{n}^{j})).$$
\end{lemma}

\begin{proof}
We already know (see Lemma \ref{lem:limitacaoPS}) that $\{v_{n}\}$ is bounded and then we can assume that $v_{n}\rightharpoonup v_{0}$
in $H^{1}_{0}(\Omega), v_{0}$ is a weak solution of \eqref{star} and $|I_{*}(v_{n})| \leq C.$
Assume that $\{v_{n}\}$ does not converges strongly to $v_{0}$.

%Let us set  $w_{n}:=v_{n} - v_{0}\rightharpoonup0 $; we have in particular $w_{n}\to 0$ in $L^{2}(\Omega)$ 

\medskip

Let $w_{n}^{1}: = v_{n} - v_{0} \rightharpoonup 0$. Then by Lemma \ref{lem:limitacaoPS},
$\{w_{n}^{1}\}$ is a $(PS)$ sequence for $I_{*}$  and
\begin{equation}\label{eq:PS1}
I_{*}(v_{n}) = I_{*}(v_{0})+I_{*}(w_{n}^{1}) +o_{n}(1).
\end{equation}

By Lemma \ref{lem:preliminare} applied to $\{w_{n}^{1}\}$,
we get the existence of sequences $\{x_{n}^{1}\}\subset \Omega, \{R_{n}^{1}\}\subset(0,+\infty)$ 
with $R_{n}^{1} \to +\infty$ and $v^{1}\in D^{1,2}(\mathbb R^{N})$ solution of \eqref{star}, such that
\begin{itemize}
\item[(1a)] defining $ w_{n}^{2}:=w_{n}^{1} - v^{1}_{R_{n}^{1} , x_{n}^{1}} +o_{n}(1)$
%(R_{n}^{1})^{N/2^{*}} v^{1}(R_{n}^{1} (\cdot - x_{n}^{1}))+o_{n}(1)$
with $o_{n}(1)\to 0$ in $D^{1,2}(\mathbb R^{N})$, and $\{w_{m}^{2}\}$ is a $(PS)$
 sequence for $I_{*}$, \medskip
\item[(1b)] $w^{2}_{n}\rightharpoonup 0$ in $H^{1}_{0}(\Omega)$, \medskip
\item[(1c)] $I_{*}(w_{n}^{2}) = I_{*} (w_{n}^{1}) - \widehat I(v^{1}) + o_{n}(1)$, \medskip
\item[(1d)] $R_{n}^{1} d(x_{n}^{1} , \partial \Omega) \to +\infty$, \medskip
\item[(1e)] if $I_{*}(w_{n}^{1}) \to \beta<\beta^{*}$, then $\{w_{n}^{1}\}$ is relatively compact;
in particular $w_{n}^{1}\to 0$ in $H^{1}_{0}(\Omega)$ and $ I_{*}(w_{n}^{1})\to 0$.
\end{itemize}
Then by  (1c) equation \eqref{eq:PS1} becomes
\begin{equation}\label{eq:inducao1}
I_{*}(v_{n}) = I_{*}(v_{0}) + I_{*}(w_{n}^{2}) +\widehat I(v^{1}) +o_{n}(1).
\end{equation}
Note that, by definitions, $w_{n}^{2} = v_{n} - v_{0} -v^{1}_{R_{n}^{1}, x_{n}^{1}}+o_{n}(1)$.
Hence, if $\{w_{n}^{2}\}$  is strongly convergent to zero, the Theorem is proved with $k=1$.
Otherwise,  in virtue of $(1a)$ and $(1b)$, we can apply  Lemma \ref{lem:preliminare} to the sequence
$\{w_{n}^{2}\}$: then
we get the existence of sequences $\{x^{2}_{n}\}\subset \Omega, \{R_{n}^{2}\}\subset (0,+\infty)$
with $R_{n}^{2}\to+\infty$
 and $v^{2}\in D^{1,2}(\mathbb R^{N})$ solution of \eqref{star}, such that
\begin{itemize}
\item[(2a)] $ w_{n}^{3}:=w_{n}^{2} - v^{2}_{R_{n}^{2},x_{n}^{2}} + o_{n}(1)$
%(R_{m}^{2})^{N/2^{*}} v^{2}(R_{m}^{2} (\cdot - x_{m}^{2}))+o_{n}(1)$
with $o_{n}(1)\to 0$ in $D^{1,2}(\mathbb R^{N})$, and $\{w_{m}^{3}\}$ is a $(PS)$ sequence for $I_{*}$,\medskip
\item[(2b)] $w^{3}_{m}\rightharpoonup 0$ in $H^{1}_{0}(\Omega)$, \medskip
\item[(2c)] $I_{*}(w_{n}^{3}) = I_{*} (w_{n}^{2}) - \widehat I(v^{2}) + o_{n}(1)$, \medskip
\item[(2d)] $R_{n}^{2} d(x_{n}^{2} , \partial \Omega) \to +\infty$, \medskip
\item[(2e)] if $I_{*}(w_{n}^{2}) \to \beta<\beta^{*}$, then $\{w_{n}^{2}\}$ is relatively compact;
in particular $w_{n}^{2}\to 0$ in $H^{1}_{0}(\Omega)$ and  $I_{*}(w_{n}^{2})\to 0$.
\end{itemize}
Then by \eqref{eq:inducao1} and (2c):
\begin{equation*}\label{eq:paso3}
I_{*}(v_{n}) =I_{*}(v_{0})+ I_{*}(w_{n}^{3}) + \widehat I(v^{1}) + \widehat I(v^{2}) +o_{n}(1).
\end{equation*}
It is $w_{n}^{3} =v_{n} - v_{0} -v^{1}_{R_{n}^{1}, x_{n}^{1}} +v^{2}_{R_{n}^{2}, x_{n}^{2}} +o_{n}(1)$.
If $\{w_{n}^{3}\}$  is strongly convergent to zero, the Theorem is proved with $k=2$,
otherwise we go on.

By arguing in this way, at the $j-$th stage ($j>1$) we have: $w_{n}^{j-1} \rightharpoonup 0$ in $H^{1}_{0}(\Omega)$ and we get the existence of sequences $\{x^{j-1}_{n}\}\subset \Omega, \{R_{n}^{j-1}\}\subset (0,+\infty)$
with $R_{n}^{j-1}\to+\infty$ and $v^{j-1}\in D^{1,2}(\mathbb R^{N})$ solution of \eqref{star},
such that
\begin{itemize}
\item[(ja)] $ w_{n}^{j}:=w_{n}^{j-1} - v^{j-1}_{R_{n}^{j-1} , x_{n}^{j-1}} + o_{n}(1)$
% (R_{n}^{j-1})^{N/2^{*}} v^{j-1}(R_{n}^{j-1} (\cdot - x_{n}^{j-1}))+o_{n}(1)$
with $o_{n}(1)\to 0$ in $D^{1,2}(\mathbb R^{N})$, and $\{w_{n}^{j}\}$ is a $(PS)$ sequence for $I_{*}$,\medskip
\item[(jb)] $w^{j}_{n}\rightharpoonup 0$ in $H^{1}_{0}(\Omega)$, \medskip
\item[(jc)] $I_{*}(w_{n}^{j}) = I_{*} (w_{n}^{j-1}) - \widehat I(v^{j-1}) + o_{n}(1)$, \medskip
\item[(jd)] $R_{n}^{j-1} d(x_{n}^{j-1} , \partial \Omega) \to +\infty$, \medskip
\item[(je)] if $I_{*}(w_{n}^{j-1}) \to \beta<\beta^{*}$, then $\{w_{n}^{j-1}\}$ is relatively compact;
in particular $w_{n}^{j-1}\to 0$ in $H^{1}_{0}(\Omega), I_{*}(w_{n}^{j-1})\to 0$.
\end{itemize}
%\begin{equation*}\label{eq:generale}
%I_{*}(v_{n}) = I_{*}(v_{0})+I_{*}(w_{n}^{j}) + \sum_{i=1}^{j-2}\hat I(v^{i}) +o_{n}(1)
%\end{equation*}
As before it is
\begin{equation}\label{eq:jgenerale0}
w_{n}^{j} = v_{n}-v_{0} -\sum _{i=1}^{j-1}v^{i}_{R_{n}^{i} , x_{n}^{i}},
\end{equation}
and by (jc) we have
\begin{equation}\label{eq:jgenerale}
I_{*}(v_{n}) = I_{*}(v_{0})+I_{*}(w_{n}^{j}) + \sum_{i=1}^{j-1}\widehat I(v^{i}) +o_{n}(1).
\end{equation}

\medskip

Recalling that $I_{*}(v_{0}) \geq 0$  the previous identity gives
%{\bf Claim:} Setting 
%$$\hat{\mathcal N} = \left\{ v\in D^{1,2}(\mathbb R^{N})\setminus\{0\} : \hat{I}' (v)[v] = 0\right\} \quad
%\mbox{ and } \quad \hat{m} = \inf_{v\in \hat{\mathcal N}} \hat I(v)$$ it holds
%$$I_{*}(v_{0}) \geq 0 \quad  \mbox{and} \quad \hat{I}(v^{i}) \geq \hat m>0, \quad \forall i=1,\ldots,j-1.$$
%and then by \eqref{eq:generale},
\begin{equation}\label{eq:quasifinale}
C\geq I_{*}(v_{n}) \geq I_{*}(w_{n}^{j}) +(j-1)\widehat m +o_{n}(1).
\end{equation}
On the other hand, since $\{w_{n}^{j}\}$ is a bounded $(PS)$ sequence for $I_{*}$,
\begin{eqnarray*}
I_{*}(w_{n}^{j}) &= &I_{*}(w_{n}^{j}) -\frac{1}{2^{*}} I'_{*}(w_{n}^{j})[w_{n}^{j}] +o_{n}(1) \\
&=& \frac{2^{*} - 2}{22^{*}}\int_{\Omega}|\nabla w_{n}^{j}|^{2}
+\frac{1}{2^{*}}\int_{\Omega}|f(w_{n}^{j})|^{22^{*} -2}f(w_{n}^{j}) f'(w_{n}^{j}) w_{n}^{j} 
-\frac{1}{22^{*}}\int_{\Omega}|f(w_{n}^{j})|^{22^{*}} +o_{n}(1)\\
&\geq& \frac{2^{*} - 2}{22^{*}}\int_{\Omega}|\nabla w_{n}^{j}|^{2} +o_{n}(1)\\
&\geq& o_{n}(1)
\end{eqnarray*}
so that, by \eqref{eq:quasifinale}, being $\widehat m>0$, we deduce
 that the process has to finish after a finite number of steps, let us say at some index $k$.
This means, see \eqref{eq:jgenerale0}, that
$$w_{n}^{k+1} = v_{n} - v_{0}  -\sum_{i=1}^{k} v_{R_{n}^{i}, x_{n}^{i}}^{i} \to 0,$$
giving \eqref{eq:approxv0}. %and so $\lim_{n\to+\infty}I_{*}(w_{n}^{k+1}) = 0$.
Moreover as in \eqref{eq:jgenerale} it is
$$I_{*}(v_{n}) = I_{*}(v_{0}) + I_{*}(w_{n}^{k+1}) +\sum_{i=1}^{k}\widehat I(v^{i}) +o_{n}(1)$$
and we deduce \eqref{eq:approxI}, concluding the proof.
\end{proof}

Now, it is known that there exists $U$ solution of 
\begin{equation*}
\left\{
\begin{array}[c]{ll}
-\Delta u =|u|^{2^{*}-2}u\medskip  \\ 
u\in D^{1,2}(\mathbb R^{N})
\end{array}
\right.
\end{equation*}
such that $\widehat I(U_{R, \xi}) = m_{*}$ (recall the definitions in \eqref{eq:m*} adapted to the case $\Omega=\mathbb R^{N}$ and \eqref{eq:conformal}) 
and that on any other solution $W$ which is not of this type, it is $\widehat I (W)\geq 2m_{*}$.
Then, setting $V_{R,\xi} = f^{-1}(U_{R,\xi})$,
it is also $\widehat I(V_{R,\xi}) = m_{*}$ and, on any other solution $Z$ of \eqref{eq:problRN}
which do not belong to the family $\{V_{R,\xi}\}_{R,\xi}$ ,
it holds $\widehat I(Z)\geq 2 m_{*}$.

By this observation, we deduce that
  if $\{ v_{n}\}$ is a $(PS)$ sequence for $I_{*}$ at level $m_{*}$
and $v_{n}\rightharpoonup v_{0}$,
Lemma \ref{lem:splitting} gives, $v_{n}\to v$ in $H^{1}_{0}(\Omega)$,
and in this case we have compactness, 
or equivalently, the Lemma holds with $k=1$. In this case
$$m_{*} =  I_{*}(v_{0}) + \widehat I(v^{1})+o_{n}(1)$$
and since $I_{*}(v_{0})\geq 0$, it has to be necessarily $v^{0}=0, v^{1} = V$; therefore
 $$v_{n} = V_{R_{n},x_{n}} +o_{n}(1) \quad \text{ in }\quad D^{1,2}(\mathbb R^{N}).$$
This final observation will be used below.

\section{The barycenter map}	\label{bary}
The aim of this section is to localize the barycenters of functions on $\mathcal N_{p}$
which are almost at the ground state level.
Indeed,  thanks to the results proved in the previous sections, we are able to show that,
roughly speaking, the
functions in the Nehari manifold $\mathcal N_{p}$ (at least for $p$ near the critical exponent 
$22^{*}$) which are almost at the
ground state level $\mathfrak m_{p}$, have barycenter ``near'' $\Omega$. This is the main result of this Section
(see Proposition \ref{baricentri})
and will be fundamental in the next Section in order to prove the multiplicity results for our problem.

We begin by introducing the barycenter map that will
allow us to compare the topology of $\Omega$ with the topology of suitable
sublevels of $I_{p}\,;$ precisely sublevels with energy near the minimum level $\mathfrak m_p.$

For $u\in H^{1}(\mathbb{R}^{N})$ with compact support, 
let us denote with the same symbol $u$ its
trivial extension out of supp\,$u$.
In particular a function in  $H^{1}_{0}(\Omega)$ can be thought 
also as an element of $D^{1,2}(\mathbb{R}^{N})$.

 The barycenter of $u$ (see \cite{BC}) is defined as
\begin{equation*}
\beta(u)=\dfrac{\displaystyle\int_{\mathbb{R}^{N}} x|\nabla u|^{2}}{\displaystyle\int_{\mathbb{R}
^{N}}|\nabla u|^{2}}\in \mathbb R^{N}.
\end{equation*}
From now on, we fix  $r>0$ a radius sufficiently small such that $B_{r}\subset\Omega$ and the sets
$$
\Omega^{+}_{r}=\{x\in\mathbb{R}^{3}:d(x,\Omega)\le r\}\,
$$
$$
\Omega^{-}_{r}=\{x\in\Omega:d(x,\partial\Omega)\ge r\}\,
$$
are homotopically equivalent to $\Omega $. 
%Of course we are assuming, without loss of generality that $0\in \Omega$ and
 $B_{r}$ stands for the ball of radius $r>0$ centred in $0$.
We denote by
\begin{equation}\label{h}
h:\Omega^{+}_r \rightarrow\Omega^-_r
\end{equation}
the homotopic equivalence map such that $h|_{\Omega^-_r}$ is the identity.

\smallskip

Now we have the following:
\begin{proposition}\label{baricentri}
There exists $\varepsilon>0$ such that if $p\in(22^{*}-\varepsilon,22^{*}),$ it
follows
$$v\in\mathcal{N}_{p} \ \mbox{ and } \ I_{p}(v)<\mathfrak m_{p}+\varepsilon\,
\Longrightarrow\,\beta(v)\in\Omega^{+}_{r}.$$
\end{proposition}
\begin{proof}
We argue by contradiction. Assume that there exist sequences $\varepsilon
_{n}\rightarrow0,p_{n}\rightarrow 22^{*}$ and $w_{n} \in\mathcal{N}_{p_{n}}$ such
that
\begin{equation}\label{contradiction}
\mathfrak m_{p_{n}}\leq I_{p_{n}}(w_{n})\le \mathfrak m_{p_{n}}+\varepsilon_{n}
\  \mbox{ and } \ \ \beta(w_{n})\notin\Omega^{+}_{r}.
\end{equation}
Then by Theorem \ref{prop:limitem*} we deduce
\begin{equation}\label{converg}	
	I_{p_{n}}(w_{n})\rightarrow m_{*}
\end{equation}
and then by Remark \ref{rem:general}, $\{w_n\}$ is bounded in $H^{1}_{0}(\Omega)$.
We can suppose that $w_{n}\rightharpoonup w$ in $H^{1}_{0}(\Omega)$.
Since all the Nehari manifolds  $\mathcal N_{p}$ are bounded away from zero 
(see Lemma \ref{lemmanehari} and Remark \ref{rem:limbaixo}) we know that $w_{n}\not\to0$ in $H^{1}_{0}(\Omega)$
and then, by Remark \ref{rem:nuovo}, we deduce
$|w_{n}|_{2^{*}} \not\to 0.$

%% Aqui estava se provando que nao converge fraco para zero:
%mas nao precisa
%Moreover, if $w_{n}\rightharpoonup 0$ then 
%\textcolor{red}{CONFERIR ESSA CONVERGENCIA: deve seguir do fato de ser subcritico,
%ou usamos Claim1 abaixo??}
%$$0< c\leq \|w_{n}\|^{2}=\int_{\Omega} |f(w_{n})|^{p_{n}-2} f(w_{n}) f'(w_{n})w_{n}
% \leq 2^{p_{n}/4}\int_{\Omega} |w_{n}|^{p_{n}/2}\to 0$$
%which gives a contradiction. Hence $w_{n}\rightharpoonup w\neq0$.
Since the functions $|f(t)|^{p-2}, f(t)f'(t)t$ are even, it is 
$I'_{p}(v)[v] = I'_{p}(|v|)[|v|]$; hence we can assume, without loss of generality, that $w_{n}\geq0$.

Let $t_{*}(w_{n})>0$ such that $t_{*}(w_{n}) w_{n}\in\mathcal{N}_{*}$.
By Proposition  \ref{prop:importante} we have $\lim_{n\to+\infty}t_{*}(w_{n})= 1$.

\smallskip

The proof now consists in
\begin{itemize}
\item {\bf STEP 1:} prove that $\{t_{*}(w_{n})w_{n}\}\subset \mathcal N_{*} $ is  a minimizing sequence for $I_{*}$ on $\mathcal N_{*}$;
\item {\bf STEP 2:} use the Ekeland Variational Principle and write $t_{*}(w_{n})w_{n}=V_{R_{n},x_{n}}+z_{n}$ where
$V_{R_{n},x_{n}}$ is introduced at the end of Section \ref{sec:Nehari} and $z_{n}\to 0$ in $D^{1,2}(\mathbb R^{N})$;
\item {\bf STEP 3:} compute the barycentre of $t_{*}(w_{n})w_{n}$ by using the representation obtained in STEP 2
and contradict \eqref{contradiction}, finishing the proof of the proposition.
\end{itemize}

\medskip

\medskip

{\bf STEP 1:} $\lim_{n\to+\infty} I_{*}(t_{*}(w_{n})w_{n})=m_{*}. $

\medskip

%{\bf Claim 3:} $I_{*}(t_{n}w_{n}) - I_{p_{n}}(w_{n})\leq o_{n}(1)$.

Observe that by the H\"older inequality, \eqref{outrasiv} and \eqref{outrasvii} of Lemma \ref{lem:outras}
one has:
\begin{eqnarray*}%\label{eq:}
I_{*}(t_{*}(w_{n})w_{n}) - I_{p_{n}}(w_{n})&=&\frac{t_{*}(w_{n})}{2}\|w_{n}\|^{2}-\frac{1}{22^{*}}\int_{\Omega}
f(t_{*}(w_{n})w_{n})^{22^{*}}
-\frac{1}{2}\|w_{n}\|^{2}+\frac{1}{p_{n}}\int_{\Omega}f(w_{n})^{p_{n}}\\
&\leq&\frac{t_{*}(w_{n})^{2}-1}{2}\|w_{n}\|^{2}-\frac{\tau_{n}}{22^{*}}\int_{\Omega}f(w_{n})^{22^{*}}\\
&+&\frac{1}{p_{n}}|\Omega|^{\frac{22^{*} - p_{n}}{22^{*}}} \Big( \int_{\Omega} f(w_{n})^{22^{*}}\Big)^{p_{n}/22^{*}}
\end{eqnarray*}
where $\tau_{n}:=\max \{t_{*}(w_{n})^{2^{*}}, t_{*}(w_{n})^{22^{*}}\}$. 
Then passing to the limit in $n$,  by using that $\lim_{n\to+\infty} t_{*}(w_{n})=1$, that
$\{w_{n}\}$  is bounded and that $\int_{\Omega} f(w_{n})^{22^{*}}\to M>0$, we infer
$I_{*}(t_{*}(w_{n})w_{n}) - I_{p_{n}}(w_{n})\leq o_{n}(1)$. Then
$$
0<m_{*}\le I_{*}(t_{*}(w_{n}) w_{n})  \leq I_{p_{n}}(w_{n})+o_{n}(1)
%\\%\left( \frac{1}{p_{n}}-\frac{1}{2^{*}}\right) t_{n}^{2} \|u_{n}\|^{2}\\
$$
and by \eqref{converg} we conclude $I_{*}(t_{*}(w_{n}) w_{n})\rightarrow m_{*}$ for $n\rightarrow+\infty$.

\bigskip

{\bf STEP 2:} Representation of the minimizing sequence $\{{t_{*}(w_{n})w_{n}}\}$.

\medskip

Since $\{t_{*}(w_{n})w_{n}\}$ is a minimizing sequence for $I_{*}$,
 the Ekeland's Variational Principle implies that there exist $\{v_{n}\}\subset
\mathcal{N}_{*} $ and $\{\mu_{n}\}\subset\mathbb{R}$, a sequence of Lagrange multipliers, such that
\begin{align*}
& \|t_{*}(w_{n}) w_{n} -v_{n}\|\rightarrow0\\
& I_{*}(v_{n})\rightarrow m_{*}\\
& I_{*}^{\prime}(v_{n})-\mu_{n} G_{*}^{\prime}(v_{n})\rightarrow0\nonumber
\end{align*}
and Lemma \ref{lemmaPS} ensures that $\{v_{n}\}$ is a $(PS)$ sequence for the free functional
$I_{*}$ on the whole space $H^{1}_{0}(\Omega)$ at level $m_*$.
By the arguments at the end of Section \ref{sec:Nehari} we have
$$
v_{n} - V_{R_{n}, x_{n}}\rightarrow0\ \ \ \text{ in } D^{1,2}
(\mathbb{R}^{3})
$$
where $\{x_{n}\}\subset\Omega, R_{n}\rightarrow+\infty$. Then  we can write
$$v_{n}=V_{R_{n}, x_{n}}+z_{n}$$
with a remainder $z_{n}$ such that $\|z_{n}\|_{D^{1,2}(\mathbb{R}^{N}
)}\rightarrow0 $\,.
It is clear that $$t_{*}(w_{n}) w_{n}=v_{n}+ t_{*}(w_{n}) w_{n}-v_{n}= v_{n}+o_{n}(1);$$
 so, renaming the
remainder again $z_{n}$, we have
$$t_{*}(w_{n})w_{n}=V_{R_{n},x_{n}}+z_{n}.$$

\smallskip

{\bf STEP 3:} Computing the barycenter and finishing the proof.
\medskip

By using the representation obtained in STEP 2, the $i-$th
coordinate of the barycenter of $t_{*}(w_{n})w_{n}$ satisfies
\begin{equation}\label{bari} 
\beta(w_{n})^{i} \| t_{*}(w_{n}) w_{n}\|^{2}_{D^{1,2}(\mathbb{R}^{N})}
= \int_{\mathbb{R}^{N}}x^{i} |\nabla V_{R_{n},x_{n}}|^{2}\\ 
            +\int_{\mathbb{R}^{N}} x^{i}|\nabla z_{n}|^{2}
+2\int_{\mathbb{R}^{N}}x^{i} \nabla V_{R_{n},x_{n}}\nabla z_{n}
\end{equation}
where $x^{i}$ is the $i-$th coordinate of $x\in\mathbb{R}^{N}$.
In order to localise the barycenters we need to pass to the limit in each term in the above expression;
however, at this stage, the computation of each term is 
completely analogous to the estimates made in \cite[pag. 296-7]{Sicilia}:
it just involves changes of variables in the integrals.
We just recall here the final results: it is
\begin{eqnarray*}
\| t_{*}(w_{n}) w_{n}\|^{2}_{D^{1,2}(\mathbb{R}^{N})}&=&\|V\|^{2}
_{D^{1,2}(\mathbb{R}^{N})}+o_{n}(1),\\
\int_{\mathbb R^{N}} x^{i}|\nabla V_{R_{n},x_{n}}|^{2} &=& x_{n}^{i} \int_{\mathbb R^{N}} |\nabla V|^{2},  \\
\int_{\mathbb R^{N}} x^{i}|\nabla z_{n}|^{2}& = & \int_{\mathbb R^{N}} x^{i}\nabla V_{R_{n}, x_{n}}\nabla z_{n} =o_{n}(1).
\end{eqnarray*}
Then by \eqref{bari} we find for the $i-th$ coordinate of the barycenter,
\begin{equation*}
\label{b}\beta(w_{n})^{i}= \frac{x_{n}^{i} \displaystyle\int_{\mathbb{R}^{N}}|\nabla V|^{2}+o_{n}(1)}
{\|V\|^{2}_{D^{1,2}(\mathbb{R}^N)}+o_{n}(1)}\,.
\end{equation*}
Since $\{x_{n}\}\subset\Omega$ the above equation implies that for large $n$ is $\beta(w_{n})\in\overline{\Omega
}$: this is in contrast with \eqref{contradiction} and  proves the proposition.
\end{proof}

\section{Proof of Theorem \ref{th:main}}\label{sec:finale}
Here we complete the  proof of our theorem but first we need a slight modification to the
previous notations.
Let $r>0$ be the one fixed at the beginning of Section \ref{bary}, that is
in such a way that
$\Omega^{+}_{r}=\{x\in\mathbb{R}^{3}:d(x,\Omega)\le r\}\,
$
and 
$
\Omega^{-}_{r}=\{x\in\Omega:d(x,\partial\Omega)\ge r\}\,
$
are homotopically equivalent to $\Omega $.
We  add a subscript $r$,
to denote the same quantities defined in the previous sections when the domain
$\Omega$ is replaced by $B_{r}$; namely integrals are taken on $B_{r}$ and norms are taken
for functional spaces defined on $B_{r}.$
Hence for example, for all $p\in (4,22^{*})$ we set:

$$\mathcal{N}_{p,r}=\left\{ u\in H^{1}_{0}(B_{r}):\| u\|^{2}_{H^{1}_{0}(B_{r}
)} = \int_{B_{r}} |f(v) |^{p-2}f(v)f'(v)v \right\},$$
$$I_{p,r}(v)= \frac{1}{2}\|v\|^{2}_{H^{1}_{0}(B_{r})} - \frac{1}{p}\int_{B_{r}} |f(v)|^{p},$$
$$\mathfrak m_{p,r}=\min_{v\in \mathcal{N}_{p,r}} I_{p,r}(v)=I_{p,r}(\mathfrak g_{p,r}).$$
Observe that, by means of the Palais Symmetric Criticality Principle, the ground state $\mathfrak g_{p,r}$
is radial. 
Moreover let
$$
I_{p}^{\mathfrak m_{p,r}}=\left\{ u\in\mathcal{N}_{p}: I_{p}(u)\le \mathfrak m_{p,r}\right\}
$$
which is non vacuous since $\mathfrak m_{p}<\mathfrak m_{p,r}$.

Define also, for $p\in(4,22^*)$ the map $\Psi_{p,r}:\Omega_{r}^{-}\rightarrow {\mathcal N}_{p}$ such that
\[
\Psi_{p,r}(y)(x)=
\left\{
\begin{array}
[c]{ccl}
\mathfrak g_{p,r}(\left\vert x-y\right\vert ) & \mbox{if} & x\in B_{r}(y)\\
0 & \mbox{if} & x\in\Omega\setminus B_{r}(y)
\end{array}
\right.
\]
and note that we have
\begin{equation*}
	\beta(\Psi_{p,r}(y))=y\ \ \ \mbox{and}\ \ \ \Psi_{r,p}(y)\in I_p^{\mathfrak m_{p,r}}\,.
\end{equation*}
Moreover, 
since $\mathfrak m_p+k_p=\mathfrak m_{p,r}$ where $k_p>0$ and tends to zero if $p\rightarrow22^*$
(see Theorem \ref{prop:limitem*}), in
correspondence of $\varepsilon>0$ provided by Proposition \ref{baricentri}, there exists a $\overline{p}\in [4,22^*)$ such
that for every $p\in[\overline{p},22^*)$ it results $k_p<\varepsilon$; so if $v\in I_p^{\mathfrak m_{p,r}}$ we have
$$I_p(u)\le \mathfrak m_{p,r}<\mathfrak m_p+\varepsilon,$$
at least for $p$ near $22^*.$
Hence we can define the following maps:
$$\Omega^-_r\stackrel{\Psi_{p,r}}{\longrightarrow}I_p^{\mathfrak m_{p,r}}\stackrel{h\circ\beta}{\longrightarrow}\Omega_r^-\,$$
with $h$  given by \eqref{h}.
Since the composite map $ h\circ\beta\circ\Psi_{p,r}$ is homotopic to the identity of $\Omega^-_r$
%(this is sees exactly as in \cite{BC1})
by a property of the category we have 
%the sublevel $I_p^{\mathfrak m_{p,r}}$ ``dominates'' the set $\Omega_r^-$ in the sense that
$$\mbox{cat}_{I_p^{\mathfrak m_{p,r}}}(I_p^{\mathfrak m_{p,r}})\ge \mbox{cat}_{\Omega^-_r} (\Omega^-_r)$$
%=\mbox{cat}_{\bar{\Omega}}(\bar{\Omega}).$$
and due to our choice of $r$, it follows 
$\mbox{cat}_{\Omega^-_r} (\Omega^-_r)=\mbox{cat}_{\overline{\Omega}}(\overline{\Omega})$.
Then we have found a sublevel of $I_p$ on $\cal N_p$ with category greater than 
$\mbox{cat}_{\overline{\Omega}}(\overline{\Omega})$
and since the $(PS)$ condition is verified on $\cal N_p$\,,  the Lusternik-Schnirel\-mann 
theory guarantees the existence of at least $\mbox{cat}_{\overline{\Omega}}(\overline{\Omega})$ critical points for $I_p$
on the manifold $\cal N_p$ which give rise to solutions of \eqref{principal}.

\medskip

The existence of another solution is obtained with the same arguments of Benci, Cerami and Passaseo
\cite{BCP}. We recall here the arguments for the reader convenience.
 Since by assumption
$\Omega$ is not contractible in itself, by the choice of $r$ it results $\mbox{cat}_{\Omega_{r}^{+}}\,(\Omega_{r}^{-})>1$, namely
$\Omega_{r}^{-}$ is not contractible in $\Omega_{r}^{+}$. 

\smallskip

{\bf Claim:}  the set ${\Psi_{p,r}(\Omega_{r}^{-})}$ is not  contractible in $I_{p}^{\mathfrak m_{p,r}}$.
 %(see \cite{BC2})
 
 \smallskip
 
Indeed, assume by contradiction that $\mbox{cat}_{I_{p}^{\mathfrak m_{p,r}}}\,(\Psi_{p,r}(\Omega_{r}^{-}))=1$:
this means that there exists a map $\mathcal H \in C([0,1]\times{\Psi_{p,r}(\Omega_{r}^{-})}; I_{p}^{\mathfrak m_{p,r}})$ such that
$${\mathcal H}(0,u)=u \ \ \forall u\in{\Psi_{p,r}(\Omega_{r}^{-})}\ \  \text{and}$$
$$\exists\, w\in I_{p}^{\mathfrak m_{p,r}}: {\mathcal H}(1,u)=w \ \ \forall u\in {\Psi_{p,r}(\Omega_{r}^{-})} .$$
Then $F=\Psi_{p,r}^{-1}({\Psi_{p,r}(\Omega_{r}^{-})})$ is closed, contains $\Omega_{r}^{-}$ and 
is contractible in $\Omega_{r}^{+}$ since one can  define the map 
\begin{equation*}
{\mathcal G}(t,x)=
\begin{cases}
   { \beta(\Psi_{r,p}(x))} & \text{if  $0\le t\le 1/2$}, \\ 
    \beta ({\mathcal H}(2t-1, \Psi_{p,r}(x))) &\text{if  $1/2\le t\le1$}.
\end{cases}
\end{equation*}
%where ${j}$ is the inclusion map $\Omega_{r}^{-}\rightarrow\Omega_{r}^{+}.$
%homotopy equivalence between $\beta\circ\Psi_{p,r}$ and $h$. 
Then also $\Omega_{r}^{-}$ is contractible in $\Omega_{r}^{+}$ and this gives  a contradiction.

\medskip

On the other hand we can choose a function $z\in {\mathcal N}_{p}\setminus{\Psi_{p,r}(\Omega_{r}^{-})}$ so that the cone 
$$\mathcal C=\left\lbrace \theta z+(1-\theta) u : u\in {\Psi_{p,r}(\Omega_{r}^{-})}, \theta\in [0,1]\right\rbrace $$
is compact and contractible in $H^{1}_{0}(\Omega)$ and  $0\notin{\mathcal C}$. 
For every $v\neq0$ let $t_{p}(v)$ be the unique positive number provided by 
\eqref{lemmanehariiv} in 
Lemma \ref{lemmanehari};
 it follows that if we set
$$\widehat{\mathcal C}:=\{t_{p}(v)v : v\in \mathcal C\}, \ \ \ M_{p}:=\max_{\widehat{\mathcal C}} I_{p}$$
then $\widehat{\mathcal C}$ is contractible in $I_{p}^{M_{p}}$ and $M_p>\mathfrak m_{p,r}.$ As a consequence
also ${\Psi_{p,r}(\Omega_{r}^{-})}$ is contractible in $I_{p}^{M_{p}}.$ 

In conclusion the set $\Psi_{p,r}(\Omega_{r}^{-}) $ is contractible in $I_{p}^{M_{p}}$ and not in $I_{p}^{\mathfrak m_{p,r}}$
%It follows that the sublevel sets $I_{p}^{m_{p,r}}$ and $I_{p}^{M_{p}}$ have different topology.
and this is possible, since the $(PS)$ condition holds,  only if there is
 another critical point with critical level between $\mathfrak m_{p,r}$ and $M_{p}$.

\medskip

It remains to prove that these solutions are positive. Note that we can apply all the previous
machinery replacing the functional $I_{p}$ with
\begin{equation*}%\label{I+}
I^+_{p}(u)=\frac{1}{2}\|v\|^2 -\frac{1}{p} \int_{\Omega} |f(v^{+})|^{p-2} f(v^{+}) f'(v^{+})v^{+}
\end{equation*}
where $v^{+}= \max\{v,0\}$. Then we obtain again
at least $\mbox{cat}_{\overline{\Omega}}(\overline{\Omega})$  (or $\mbox{cat}_{\bar{\Omega}}(\overline{\Omega})+1$)
nontrivial solutions that now are positive by 
the maximum principle.
% (escrever uma referencia?).

\section{Proof of Theorem \ref{th:mainMorse}}\label{sec:Morse}
Before prove the theorem we first recall some basic facts of Morse theory
and fix some notations.

For a pair of topological spaces $(X,Y)$,
$Y\subset X,$ let $H_{*}(X,Y)$ be its singular homology with coefficients in some field $\mathbb F$
(from now on omitted) and 
$$\mathcal P_{t}(X,Y)=\sum_{k}\dim H_{k}(X,Y)t^{k}$$
the Poincar\'e polynomial of the pair. If $Y=\emptyset$, it will be always omitted in the objects which involve the pair.
Recall that  if $H$ is an Hilbert space,  $I:H\to \mathbb R$  a $C^{2}$ functional  and 
$v$ an isolated critical point  with $I(v)=c$, the  {\sl polynomial Morse index} of $v$ is
$$\mathcal I_{t}(v)=\sum_{k}\dim C_{k}(I,v)  t^{k}$$
where $C_{k}(I,v)=H_{k}(I^{c}\cap U, (I^{c}\setminus\{v\})\cap U)$ are the critical groups.
Here $I^{c}=\{v\in H: I(v)\leq c\}$ and $U$ is a neighborhood of the critical point $u$.
The multiplicity of $v$ is the number $\mathcal I_{1}(v)$.

It is known that  for a non-degenerate critical point $v$
(that is, the selfadjoint operator associated to $I''(v)$
is an isomorphism)
it is $\mathcal I_{t}(v)=t^{\mathfrak i (v)}$,
where $\mathfrak i(v)$ is the {\sl (numerical) Morse index of $v$}: the maximal dimension
of the subspaces where $I''(v)[\cdot,\cdot]$ is negative definite.

\medskip

%\subsection{Proof of the Theorem}

Coming back to our functional, 
it is straightforward to see that  $I_{p}$ is of class $C^{2}$ and
for $v,w,u\in H^{1}_{0}(\Omega)$:
\begin{equation*}\label{I''}
I_{p}''(v)[w,u]=\int_{\Omega}\nabla w \nabla u-(p-1)\int_{\Omega}|f(v)|^{p-2} (f'(v))^{2}wu-
\int_{\Omega}|f(v)|^{p-2}f(v)f''(v)wu.
\end{equation*}
Hence $I_{p}''(v)$ is represented by the operator
\begin{equation*}
L_{p}(v):=\textrm R(v)-\mathrm K_{p}(v): H^{1}_{0}(\Omega)\to H^{-1}(\Omega)
\end{equation*}
where $\mathrm R(v)$ is the Riesz isomorphism
$$\mathrm R(v): w\in H^{1}_{0}(\Omega) \mapsto \mathrm R(v)[w]\in H^{-1}(\Omega)$$
acting as
$$\forall u\in H^{1}_{0}(\Omega): (\mathrm R(v)[w])[u] =\int_{\Omega} \nabla w\nabla u $$
and 
$$\mathrm K_{p}(v): w\in H^{1}_{0}(\Omega) \mapsto \mathrm K_{p}(v)[w] \in  H^{-1}(\Omega)$$
acts as
$$\forall u\in H^{1}_{0}(\Omega) : (\mathrm K_{p}(v)[w] )[u] = (p-1)\int_{\Omega}|f(v)|^{p-2} (f'(v))^{2}wu-
\int_{\Omega}|f(v)|^{p-2}f(v)f''(v)wu.$$
\begin{lemma}
The operator $\mathrm K(v)$ is  compact, that is, if
 $w_{n}\rightharpoonup 0$ then $\|\mathrm K(v)[w_{n}]\|\to 0$ in  $H^{-1}(\Omega)$.
\end{lemma}
\begin{proof}
Indeed for every  $u\in H^{1}_{0}(\Omega)$,
by \eqref{Lemafii} and \eqref{Lemafv} of Lemma \ref{Lema f}, 
 \begin{eqnarray}\label{eq:1o}
\Big| \int_{\Omega}|f(v)|^{p-2}(f'(v))^{2}w_{n} u \Big| &\leq& 2^{1/4} \int_{\Omega}|v|^{(p-2)/2}|w_{n} u|  \nonumber\\
&\leq & 2^{1/4} \Big| |v|^{(p-2)/2}\Big| _{22^{*}/(p-2)} |w_{n}| _{22^{*}/p} |u|_{22^{*}/(22^{*}-(2p-2))} \nonumber \\
&\leq& C |v|^{(p-2)/2}_{2^{*}} | w_{n}|_{22^{*}/p} \|u\| \longrightarrow 0 \nonumber
\end{eqnarray}
being
$$\frac{p-2}{22^{*}}+\frac{p}{22^{*}} + \frac{22^{*}-(2p-2)}{22^{*}}=1, \quad \frac{22^{*}}{p}\in(1, 2^{*}) \quad  \text{ and } 
\quad \frac{22^{*} }{22^{*}- (2p-2)}\in (1, 2^{*}].$$

The second integral in $\mathrm K(v)$ can be reduced to the first one.
Indeed, by using first that  $|f''(t)| = 2|f(t)| |f'(t)|^{4}$ (see the proof of \eqref{monotonicidadeii} of Corollary \ref{monotonicidades}) and then \eqref{Lemafix} of Lemma \ref{Lema f}, we get
%again by  \eqref{Lemafii} and \eqref{Lemafv} of Lemma \ref{Lema f}, we get
\begin{eqnarray*}%\label{eq:2o}
\Big|\int_{\Omega}|f(v)|^{p-2}f(v)f''(v)w_{n}u \Big| &\leq&  \int_{\Omega}|f(v)|^{p-2} f^{2}(v)(f'(v))^{4} |w_{n}u| 
\nonumber \\
&\leq& \frac{1}{2} \int_{\Omega} |f(v)|^{p-2} (f'(v))^{2} |w_{n} u| \nonumber,
%&\leq& 2^{1/4}  \textcolor{red}{\Big| |v|^{p/2}\Big| _{\alpha} |w_{n}| _{\beta} |u|_{\gamma}} \to 0
 \end{eqnarray*}
 concluding as before. Then we  deduce that
$$\|\mathrm K_{p}(v)[w_{n}]\|=\sup_{\|u\| =1}\Big| (K_{p}(v)[w_{n}] )[u]  \Big|\rightarrow 0$$
and the proof is completed.
\end{proof}

Now  for $a\in(0,+\infty]$, let us define the sets
$$I_{p}^{a}:=\Big\{v\in H_{0}^{1}(\Omega): I_{p}(v)\leq a\Big\}\ , \qquad 
\mathcal N_{p}^{a}:= \mathcal N_{p}\cap  I_{p}^{a}$$\smallskip
$$\mathfrak K_{p}:=\Big\{v\in H^{1}_{0}(\Omega): I'_{p}(v)=0\Big\}\ ,\qquad
\mathfrak K_{p}^{a}:= \mathfrak K_{p}\cap  I_{p}^{a}\ ,\qquad
(\mathfrak K_{p})_{a}:=\Big\{v\in \mathfrak K_{p}: I_{p}(v)> a\Big\} .
$$

%Recall by Proposition \ref{prop:I-K} that $I''(v)=\textrm R(v)-\mathrm K_{p}(v)$
%where $\textrm R(v)$ is the Riesz isomorphism and $\mathrm K_{p}(v)$
%iscompact.

In the remaining part of this section we will follow \cite{BC}.
Let $\overline p$  as in Section \ref{sec:finale}
and let $p\in [\overline p,22^{*})$ be fixed.
In particular $I_{p}$ satisfies the Palais-Smale condition.
 We are going to prove that $I_{p}$ restricted to $\mathcal N_{p}$
 has at least $2\mathcal P_{1}(\Omega)-1$ critical points.
% Then Theorem \ref{gioGae2} will follow by Corollary \ref{C1}.
 \medskip
 
 We can assume,
 of course,  that there exists a regular value $b^{*}_{p}>\mathfrak m_{p,r}$
 for the functional $I_{p}$ and then 
 $$\Psi_{p,r}: \Omega_{r}^{-}\to  \mathcal N^{\mathfrak m_{p,r}}_{p}\subset \mathcal N_{p}^{b_{p}^{*}}.$$

Since $\Psi_{p,r}$ is injective, it is easily seen that it induces
injective homomorphisms between the homology groups. Then 
$\dim H_{k}(\Omega) = \dim H_{k}(\Omega_{r}^{-})%=\dim H_{k}(M^{-})
\leq \dim H_{k}(\mathcal N_{p}^{b_{p}^{*}})$
and consequently
\begin{equation}\label{obvious1}
\mathcal P_{t}(\mathcal N_{p}^{b_{p}^{*}})=\mathcal P_{t}(\Omega)+\mathcal Q(t), \qquad \mathcal Q\in \mathbb P,
\end{equation}
where hereafter $\mathbb P$ denotes the set of polynomials with non-negative integer coefficients.
 
 \medskip

The following result is analogous to \cite[Lemma 5.2]{BC}; we omit the proof.
\begin{lemma}
Let  $\overline r\in (0,\mathfrak m_{p,r})$  and $a\in(\overline r,+\infty]$ a regular level for $I_{p}$.
Then
 \begin{eqnarray}\label{PtP}
\mathcal P_{t}(I_{p}^{a},I_{p}^{\overline r})&=&t \mathcal P_{t}(\mathcal N_{p}^{a}).
\end{eqnarray}
\end{lemma}

\medskip

In particular we have the following:
\begin{corollary}\label{t}
Let $\overline r\in (0,\mathfrak m_{p,r})$. Then
\begin{eqnarray*}
\mathcal P_{t}(I_{p}^{b_{p}^{*}},I_{p}^{\overline r})&=&t\Big(\mathcal P_{t}(\Omega)+\mathcal Q(t)\Big), \qquad \mathcal Q\in \mathbb P, \label{prima}\\
\mathcal P_{t}(H^{1}_{0}(\Omega), I_{p}^{\overline r})&=& %t\mathcal P_{t}(\mathcal N^{\infty}_{\varepsilon})
%= %\mathcal P_{t}(I^{+\infty}, I^{\delta})=t \mathcal P_{t}(\mathcal N^{+\infty})=
t.\label{seconda}
\end{eqnarray*}
\end{corollary}
\begin{proof}
The first identity follows by    \eqref{obvious1} and \eqref{PtP} by choosing  $a=b^{*}_{p}$.
The second one follows by \eqref{PtP} with $a=+\infty$ and
noticing that the Nehari manifold $\mathcal N_{p}$ is contractible in itself (see
\eqref{lemmanehariv} in Lemma \ref{lemmanehari}).
\end{proof}

To deal with critical points above the level $b^{*}_{p},$ we need also the following
result whose proof is purely algebraic and is omitted. The interested reader may consult
\cite[Lemma 5.6]{BC}.
%see also \cite[Lemma 2.4]{claudianorserginhorodrigo}.
\begin{lemma}\label{quadrato}
It holds
$$\mathcal P_{t}(H^{1}_{0}(\Omega),I_{p}^{b_{p}^{*}})=t^{2}\Big(\mathcal P_{t}(\Omega)+\mathcal Q(t)-1\Big), \qquad \mathcal Q\in \mathbb P.$$
\end{lemma}
%\begin{proof}
%The proof is purely algebraic and goes exactly as in 
%\end{proof}

As a consequence of these facts we have

\begin{corollary}\label{separato}
Suppose that 
the set $\mathfrak  K_{p}$
is discrete. Then
\begin{eqnarray*}\label{1}
\sum_{u\in \mathfrak K^{b_{p}^{*}}_{p}}\mathcal I_{t}(u)=t\Big(\mathcal P_{t}(\Omega)+\mathcal Q(t)\Big)+(1+t)\mathcal Q_{1}(t)
\end{eqnarray*}
and 
\begin{eqnarray*}
\sum_{u\in(\mathfrak K_{p})_{b_{p}^{*}}} \mathcal I_{t}(u)=t^{2}\Big(\mathcal P_{t}(\Omega)+\mathcal Q(t)-1\Big)+(1+t)\mathcal Q_{2}(t),
\end{eqnarray*}
where $\mathcal Q,\mathcal Q_{1}, \mathcal Q_{2}\in \mathbb P.$

\end{corollary}
\begin{proof}
Indeed the Morse Theory 
gives
\begin{eqnarray*}\label{1}
\sum_{u\in \mathfrak K^{b_{p}^{*}}_{p}}\mathcal I_{t}(u)=\mathcal P_{t}(I_{p}^{b_{p}^{*}}, I_{p}^{\overline r})+(1+t)\mathcal Q_{1}(t)\\
\end{eqnarray*}
and 
\begin{eqnarray*}
\sum_{u\in(\mathfrak K_{p})_{b_{p}^{*}}} \mathcal I_{t}(u)=
\mathcal P_{t}(H^{1}_{0}(\Omega),I_{p}^{b_{p}^{*}})+(1+t)\mathcal Q_{2}(t)
\end{eqnarray*}
so that, by using  Corollary \ref{t} and Lemma \ref{quadrato}, we easily conclude.
\end{proof}

Finally, by Corollary \ref{separato} we get
$$\sum_{u\in \mathfrak K_{p}}\mathcal I_{t}(u)=t\mathcal P_{t}(\Omega)+t^{2}\Big(\mathcal P_{t}(\Omega)-1\Big)+t(1+t)\mathcal Q(t)$$
for some $\mathcal Q\in \mathbb P.$ 
We easily deduce that, if the critical points of $I_{p}$
are non-degenerate, then they are  at least $2\mathcal P_{1}(\Omega)-1$,
if counted with their multiplicity.

The proof of Theorem \ref{th:mainMorse} is thereby complete.
\medskip

%%%%%%%%%%%%%%%%
%              %
% Bibliografia %
%              %
%%%%%%%%%%%%%%%%
%\chapter*{Bibliografia}

%\addcontentsline{toc}{chapter}{Bibliografia}


\begin{thebibliography}{99}

%\bibitem{ABS}
%{ G. Alberti, G. Bouchitt\'e and  P. Seppecher, }
%\newblock {\em Phase transition with the line-tension effect},
%\newblock { Arch. Rational Mech. Anal., {\bf 144} (1998), 1-46}.

%\bibitem{Alves3}
%{\sc C. O. Alves}, {\sc P.C. Carri\~ao } and {\sc E. S. Medeiros, }
%\newblock {\em Multiplicity of solutions for a class of quasilinear problem in exterior
%domains with Neumann conditions},
%\newblock { Abstract and Applied Analisys {\bf 03} (2004), 251-268.}

\bibitem{Alves}
{C.O. Alves,}
\newblock {\em Existence and multiplicity of solution
for a class of quasilinear equations},
\newblock { Adv. Nonlinear Studies {\bf 5} (2005), 73-86}.

%\bibitem{Alvesgio1}{C.O. Alves and G. M. Figueiredo, }{\em Existence and multiplicity of positive solutions to a p-Laplacian equation in $
%\mathbb{R}^{N}$}, Differential and Integral Equations {\bf 19} (2006)143--162.

\bibitem{AUG} { C.O. Alves, G.M. Figueiredo and U.B. Severo, }
{\em Multiplicity of positive solutions for a class of quasilinear problems}, Adv.
Differential Equations \textbf{9/10} (2009), 911--942.


\bibitem{AFF}
{ C.O. Alves, G.M. Figueiredo and M. Furtado, }{\em On the number of solutions of NLS
equations with magnetic fields in expanding domains}, J.
Differential Equations \textbf{251} (2011), 2534--2548.

\bibitem{CGU}
{ C.O. Alves, G.M. Figueiredo and U.B. Severo, }{\em A result of multiplicity of solutions
for a class of quasilinear equations}, Proceeding of the Edinburgh Math. Soc.
\textbf{55} (2012), 291--309.

%\bibitem{claudianorserginhorodrigo} {\sc C. O. Alves}, {\sc Rodrigo C. M. Nemer} and {\sc S. H. M. Soares}
%\newblock {\em The use of the Morse theory to estimate the number of nontrivial solutions of a nonlinear Schr\"{o}dinger with magnetic fields}, \newblock { arXiv:1408.3023v1}.

%\bibitem{COL} {\sc C. O. Alves,  and O. Miyagaki, }{\em Existence and concentration of solution for a class of fractional elliptic equation in $\mathbb R^{N}$ via penalization method, } \newblock { arXiv:1508.03964v1}.
%\bibitem{B}{ P.W. Bates, }
%\newblock {\em On some nonlocal evolution equations arising in materials science},
%\newblock { Nonlinear dynamics and evolution equations, Amer. Math. Soc., vol 48 of Fields Inst. Commun. (2006), 13-52}.

\bibitem{BC1}{  V. Benci and G. Cerami, }{\em The effect of the domain topology on the number of positive
solutions of nonlinear elliptic problems}, Arch. Rat. Mech. Anal.
{\bf 114} (1991), 79--83.

\bibitem{BCP} { V. Benci, G. Cerami and D. Passaseo, }{\em On the number of positive solutions  of some nonlinear 
elliptic problems, }
Nonlinear analysis,  
Sc. Norm. Super. di Pisa Quaderni, Scuola Norm. Sup., Pisa, 1991, 93--107.

\bibitem{BC3} {V. Benci and G. Cerami, }{\em Multiple positive solutions of some elliptic problems via the
Morse theory and the domain topology} , Calc. Var. Partial
Differential Equations {\bf 02} (1994), 29--48.

\bibitem{BC}
{ V. Benci and G. Cerami, }
\newblock {\em Multiple positive solutions of some elliptic problems via the Morse
theory and the domain topology},
\newblock { Calc. Var. Partial
Differential Equations {\bf 2} (1994), 29-48}.

%\bibitem{BCdS}
%{\sc C. Br\"andle,  E. Colorado, A. de Pablo } and {\sc U. S\'anchez,  }
%\newblock {\em A concave-convex elliptic problem involving the fractional Laplacian},
%\newblock {\em Proc. Royal Soc. Edinburgh 134A (2013), 39-71}.

%\bibitem{CRS}
%{\sc L. Caffarelli, J.M. Roquejoffre } and {\sc O. Savin, }
%\newblock {\em Nonlocal minimal surfaces},
%\newblock {\em Comm. Pure Apple. Math., {\bf 63} (2012), 1111-1144.}

%\bibitem{CV}
%{\sc L. Caffarelli } and {\sc E. Valdinoci, }
%\newblock {\em Uniform estimates and limiting arguments for nonlocal minimal surfaces},
%\newblock { Calc. Var. Partial Differential Equations {\bf 32} (2007), 1245-1260.}

%\bibitem{CG}
%{ S.-Y. A. Chang  and M. del Mar Gonz\'alez, }
%\newblock {\em Fractional Laplacian in conformal geometry},
%\newblock { Adv. Math. {\bf 226} (2011), 1410-1432.}

%\bibitem{Cheng}
%{\sc M. Cheng, }
%\newblock {\em Bound state for the fractional Schr\"odingier equation with unbounded potential},
%\newblock { J. Mathematical Physics {\bf 53} (2012), 043507-1}.

%\bibitem{CS}{\sc L. Caffarelli } and {\sc L. Silvestre}
%\newblock{ \em An extension problems related to the fractional Laplacian},
%\newblock {\em Comm. PDE 32 (2007), 1245-- 1260}.

%\bibitem{CZ}
%{\sc G. Chen} and { Y. Zheng, }
%\newblock{\em Concentration phenomenon for fractional nonlinear Schr\"odinger equations},
%\newblock {\em Commun. Pure Appl. Anal. 13 (2014) 2359-2376.}

%\bibitem{CL1} {S. Cingolani and M. Lazzo, } {\em Multiple semiclassical standing waves for  a class of nonlinear Schr\"odinger equations} , 
%Topological Methods in Nonlinear Analysis {\bf 10} (1997), 1--13. 

%\bibitem{CL2} {\sc S. Cingolani and M. Lazzo, } {\em Multiple positive solutions to nonlinear Schr\"odinger equations with competing potential functions}, J. Differential Equations {\bf 160} (2000), 118--138. 

%\bibitem{CLV}{\sc  S. Cingolani, M. Lazzo and G. Vannella, } {\em Multiplicity results for a quasilinear elliptic system via Morse theory}, 
%Communications in Contemporary Mathematics {\bf 7} (2005), 227--249. 

\bibitem{Borovskii-Galkin} A. Borovskii and A. Galkin,
\textit{Dynamical modulation of an ultrashort high-intensity laser
pulse in matter.} JETP \textbf{77}, (1983), 562-573.
%
%\bibitem{Brandi-Manus} H. Brandi, C. Manus, G. Mainfray, T. Lehner
%and G. Bonnaud, \textit{Relativistic and ponderomotive self-focusing
%of a laser beam in a radially inhomogeneous plasma.} Phys. Fluids
%\textbf{B5}, (1993), 3539-3550.

\bibitem{Brizhik} L. Brizhik, A. Eremko, B. Piette and W. J. Zakrzewski, \emph{Static
solutions of a $D$-dimensional modified nonlinear Schr{\"o}dinger
equation}, Nonlinearity \textbf{16} (2003) 1481--1497.

\bibitem{CV2} {S. Cingolani and G. Vannella, } {\em Multiple positive solutions for a critical quasilinear equation via Morse theory}, 
Ann. I. H. Poincar\'e {\bf 26} (2009), 397--413. 

\bibitem{CV3} {S. Cingolani and G. Vannella,  }
{\em On the multiplicity of positive solutions for $p-$Laplace equations via Morse theory}, 
J. Differential Equations {\bf 247} (2009), 3011--3027. 


%\bibitem{BL} H. Berestycki and P. L. Lions, \textit{Nonlinear scalar
%field equations I: existence of a ground state.} Arch. Rational
%Mech. Anal. \textbf{82}, (1983), 313-346.

%\bibitem{brezislieb}
%H. Brezis and E. H. Lieb, \textit{A relation between pointwise convergence of functions and convergence
%functionals}, Proc. Amer. Math. Soc. 8(1983)486-490.
%
%\bibitem{Chen-Sudan} X. L. Chen and R. N. Sudan, \textit{Necessary
%and sufficient conditions for self-focusing of short ultraintense
%laser pulse.} Phys. Review Letters \textbf{70}, (1993), 2082-2085.
%
\bibitem{Jeanjean-Colin} M. Colin and L. Jeanjean, \textit{Solutions
for a quasilinear Schr{\"o}dinger equation: a dual approach.} Nonlinear
Anal. \textbf{56}, (2004), 213-226.
%
%\bibitem{De Bouard-Hayashi} A. De Bouard, N. Hayashi and J. Saut,
%\textit{Global existence of small solutions to a relativistic
%nonlinear Schrondinger equation.} Comm. Math. Phys. \textbf{189},
%(1997), 73-105.

\bibitem{GGM} G.M Figueiredo, M.T.O. Pimenta and G. Siciliano, 
\textit{Multiplicity results for the fractional laplacian in expanded domains}
arXive:1511.09406.

\bibitem{SF} G.M Figueiredo and G. Siciliano, 
\textit{ Positive solutions for the fractional laplacian in the almost critical case in a bounded domain},
Nonlinear Analysis: Real World Applications {\bf 36} (2017), 89--100.

\bibitem{OMS}J.M. do \'O, O. Miyagaki and S. Soares,  \textit{Soliton solutions for quasilinear
Schr{\"o}dinger equations: the critical exponential case}, Nonlinear
Anal. \textbf{67}, (2007), 3357-3372.


\bibitem{do O-Severo}J.M. do \'O and U. B. Severo. \emph{Quasilinear Schr{\"o}dinger equations involving concave and convex nonlinearities}.  Commun. Pure Appl. Anal.  \textbf{8} (2009),
621--644.



\bibitem{Hartmann} B. Hartmann and W. J. Zakrzewski, \emph{Electrons on hexagonal lattices
and applications to nanotubes}, Phys. Rev. B \textbf{68} (2003)
184302.

\bibitem{Kosevich-Ivanov} A.M. Kosevich, B.A. Ivanov and A.S.
Kovalev, \textit{Magnetic solitons in superfluid films.} J. Phys.
Soc. Japan \textbf{50}, (1981), 3262-3267.

\bibitem{Kurihura} S. Kurihara, \textit{Large-amplitude
quasi-solitons in superfluids films.} J. Phys. Soc. Japan
\textbf{50}, (1981), 3262-3267.

\bibitem{Li-Wei}
A. Li and C. Wei, \textit{Existence and Multiplicity of Nontrivial Solutions to Quasilinear Elliptic Equations.}
Adv. Nonlinear Stud. 16 (2016), no. 4, 653--666. 

\bibitem{Liu-Wang I} J. Liu and Z. Q. Wang, \textit{Soliton
solutions for quasilinear Schr\"{o}dinger equations I}, Proc. Amer.
Math. Soc. \textbf{131}, 2, (2002), 441--448.

\bibitem{Liu-Wang II} J. Liu, Y. Wang and Z. Wang, \textit{Soliton
solutions for quasilinear Schr\"{o}dinger equations II.} J. Differential
Equations \textbf{187}, (2003), 473-493.

\bibitem{Liu-wang-wang} J. Liu, Y. Wang and Z. Q. Wang, \textit{Solutions
for Quasilinear Schr\"odinger Equations via the Nehari Method},
Comm. Partial Differential Equations, \textbf{29}, (2004) 879--901.

\bibitem{Makhankov-Fedyanin} V. G. Makhankov and V. K. Fedyanin,
\textit{Non-linear effects in quasi-one-dimensional models of
condensed matter theory.} Phys. Reports \textbf{104}, (1984), 1-86.

%\bibitem{Moser}
%J. Moser, \textit{A new proof of de Giorgi's theorem concerning the
%regularity problem for elliptic differential equations,}
%Comm. Pure Apll. Math. 13 (1960), 457-468.
\bibitem{GEdw}{E. G. Murcia and G. Siciliano, }
\textit{Positive semiclassical states for a
fractional Schr\"odinger-Poisson system}
Diff. Integral Eq. {\bf 30}, (2017), 231--258

\bibitem{Poppenberg-Schmitt-Wang} M. Poppenberg, K. Schmitt and Z.
Q. Wang, \textit{On the existence of soliton solutions to
quasilinear Schr\"odinger equations}, Calc. Var. Partial
Differential Equations \textbf{14}, (2002), 329--344.

%\bibitem{Quispel-Capel} G. R. W. Quispel and H. W. Capel,
%\textit{Equation of motion for the Heisenberg spin chain.} Phys. A
%\textbf{110}, (1982), 41-80.
%
\bibitem{Ritchie} B. Ritchie, \textit{Relativistic self-focusing and
channel formation in laser-plasma interactions.} Phys. Rev. E
\textbf{50}, (1994), 687-689.

\bibitem{RS} D.Ruiz and G. Siciliano, 
{\sl Existence of ground states for a modified nonlinear Schr\"oinger equation}, 
Nonlinearity {\bf 23}, (2010), 1221-1233.

\bibitem{severo} U.B. Severo, {\sl Estudo de uma classe de equa\c c\~oes de Schr\"odinger quase-lineares}.
PhD. dissertation, Unicamp, 2007.

\bibitem{severo2}
 U.B. Severo, \textit{Existence of weak solutions for quasilinear elliptic equations involving the p-Laplacian}. Electron. J. Differential Equations (2008), no. 56, 1-16.
 
 \bibitem{Sicilia}{G. Siciliano, }{\em Multiple positive solutions for a Schr\"odinger-Poisson-Slater system},
J. Math. Anal. Appl.  {\bf 365}, (2010), 288--299.

 \bibitem{Str} M. Struwe, Variational Methods, Applications to Partial Differential Equations and Hamiltonian Systems,
Springer, 2007.

\bibitem{Takeno-Homma} S. Takeno and S. Homma, 
{\em Classical planar Heinsenberg ferromagnet, complex scalar fields and nonlinear
excitations.} Progr. Theoret. Physics \textbf{65}, (1981), 172-189.

\bibitem{Willem} M. Willem,  {\em Minimax Theorems, } Birkauser, 1997.

\end{thebibliography}
\end{document}